\documentclass[11pt,a4paper]{article}
\usepackage{amssymb}
\usepackage{amsmath}
\usepackage{latexsym}
\usepackage{amsthm}

\newtheorem{thm}{Theorem}
\newtheorem{lem}[thm]{Lemma}
\newtheorem{prop}[thm]{Proposition}
\theoremstyle{definition}
\newtheorem{remark}{Remark}

\usepackage{xpatch}
\xpatchcmd{\proof}{\itshape}{\normalfont\proofnameformat}{}{}
\newcommand{\proofnameformat}{}

\begin{document}

\renewcommand{\proofnameformat}{\bfseries}

\begin{center}
{\Large\textbf{Optimal and typical $L^2$ discrepancy of 2-dimensional lattices}}

\vspace{10mm}

\textbf{Bence Borda}

{\footnotesize Graz University of Technology

Steyrergasse 30, 8010 Graz, Austria

Email: \texttt{borda@math.tugraz.at}}

\vspace{5mm}

{\footnotesize \textbf{Keywords:} continued fraction, quadratic irrational, rational lattice,\\symmetrization, low discrepancy, limit distribution}

\end{center}

\vspace{5mm}

\begin{abstract}
We undertake a detailed study of the $L^2$ discrepancy of rational and irrational $2$-dimensional lattices either with or without symmetrization. We give a full characterization of lattices with optimal $L^2$ discrepancy in terms of the continued fraction partial quotients, and compute the precise asymptotics whenever the continued fraction expansion is explicitly known, such as for quadratic irrationals or Euler's number $e$. In the metric theory, we find the asymptotics of the $L^2$ discrepancy for almost every irrational, and the limit distribution for randomly chosen rational and irrational lattices.
\end{abstract}

\section{Introduction}

The $L^2$ discrepancy of a finite point set $P \subset [0,1)^2$ in the unit square is defined as
\[ D_2(P) = \left( \int_{[0,1]^2} \left( B(x,y) -|P|xy \right)^2 \, \mathrm{d} x \, \mathrm{d} y \right)^{1/2} , \]
where $B(x,y)=|P \cap ([0,x) \times [0,y))|$ is the number of points of $P$ which fall in the rectangle $[0,x) \times [0,y)$. The $L^2$ discrepancy is a common measure of equidistribution, with direct applications to numerical integration; for a general introduction we refer to the monograph Drmota--Tichy \cite{DT}. A seminal result of K.\ Roth \cite{RO} states that every finite point set $P$ satisfies $D_2(P) \gg \sqrt{\log |P|}$ with a universal implied constant. This is known to be sharp, with several explicit constructions e.g.\ based on digital nets attaining the optimal order $D_2(P) \ll \sqrt{\log |P|}$, see \cite{DP}.

In this paper we undertake a detailed study of the $L^2$ discrepancy of $2$-dimensional lattices. Given $\alpha \in \mathbb{R}$ and $N \in \mathbb{N}$, we will consider the $N$-element set
\[ L(\alpha, N) = \left\{ \left( \{ n \alpha \}, \frac{n}{N} \right) \in [0,1)^2 \, : \, 0 \le n \le N-1 \right\} , \]
where $\{ \cdot \}$ denotes fractional part, and the $2N$-element set
\[ S(\alpha, N) = \left\{ \left( \{ \pm n \alpha \}, \frac{n}{N} \right) \in [0,1)^2 \, : \, 0 \le n \le N-1 \right\} . \]
Note that $L(\alpha, N)$ is the intersection of the unit square $[0,1)^2$ and the lattice spanned by the vectors $(\alpha, 1/N)$ and $(1,0)$. We call $S(\alpha, N)$ the symmetrization of $L(\alpha, N)$; more precisely, $S(\alpha, N)$ is the union of $L(\alpha, N)$ and its reflection about the vertical line $x=1/2$. We study both rational and irrational values of $\alpha$.

The equidistribution properties of $S(\alpha, N)$ and $L(\alpha, N)$, in particular their $L^2$ discrepancy, are closely related to the Diophantine approximation properties of $\alpha$. Throughout this paper, $\alpha=[a_0;a_1,a_2,\dots]$ will denote the (finite or infinite) continued fraction expansion of $\alpha$, and $p_k/q_k=[a_0;a_1,\dots, a_k]$ its convergents. In the rational case it will not matter which of the two possible expansions is chosen. Roughly speaking, we will show that for $N \approx q_K$,
\[ D_2^2 (S(\alpha, N)) \approx \sum_{k=1}^K a_k^2 \quad \textrm{and} \quad D_2^2(L(\alpha, N)) \approx \sum_{k=1}^K a_k^2 + \left( \sum_{k=1}^K (-1)^k a_k \right)^2 . \]
See Propositions \ref{parsevalprop} and \ref{simpleparsevalprop} below for a precise formulation.

Our first result characterizes all irrationals for which $S(\alpha, q_K)$ resp.\ $L(\alpha, q_K)$ attains optimal $L^2$ discrepancy as $K \to \infty$. We also consider the same problem for $S(\alpha, N)$ and $L(\alpha, N)$ as $N \to \infty$. The first equivalence below generalizes a result of Davenport \cite{DA}, who showed that $S(\alpha, N)$ attains optimal $L^2$ discrepancy whenever $\alpha$ is badly approximable, i.e.\ $a_k \ll 1$.
\begin{thm}\label{optimalirrationaltheorem} Let $\alpha =[a_0;a_1,a_2, \dots]$ be irrational. We have
\[ \begin{split} D_2 (S(\alpha,N)) \ll \sqrt{\log N} \,\, &\Longleftrightarrow \,\, D_2 (S(\alpha,q_K)) \ll \sqrt{\log q_K} \,\, \Longleftrightarrow \,\,  \frac{1}{K} \sum_{k=1}^K a_k^2 \ll 1 , \\ D_2 (L(\alpha,q_K)) \ll \sqrt{\log q_K} \,\, &\Longleftrightarrow \,\, \frac{1}{K} \sum_{k=1}^K a_k^2 \ll 1 \textrm{ and } \frac{1}{\sqrt{K}} \left| \sum_{k=1}^K (-1)^k a_k \right| \ll 1. \end{split} \]
\end{thm}

\begin{remark}\label{LalphaNremark}
We also give an almost complete answer for the unsymmetrized lattice $L(\alpha, N)$ with general $N$: under the assumption $a_k \ll \sqrt{k}/\log^2 k$, we have
\[ D_2 (L(\alpha,N)) \ll \sqrt{\log N} \,\, \Longleftrightarrow \,\, \frac{1}{K} \sum_{k=1}^K a_k^2 \ll 1 \textrm{ and } \frac{1}{\sqrt{K}} \left| \sum_{k=1}^K (-1)^k a_k \right| \ll 1. \]
In the special case of a badly approximable $\alpha$, this equivalence was observed in \cite{BI,BTY2}. Note that $K^{-1} \sum_{k=1}^K a_k^2 \ll 1$ implies that $a_k \ll \sqrt{k}$; we do not know whether the slightly stronger extra assumption $a_k \ll \sqrt{k}/\log^2 k$ can be removed.
\end{remark}

More precise results can be deduced for an irrational $\alpha$ whose continued fraction expansion is explicitly known. The most interesting case is that of quadratic irrationals, whose continued fractions are of the form $\alpha=[a_0;a_1,\dots, a_r,\overline{a_{r+1}, \dots, a_{r+p}}]$, where the overline denotes the period. Note that in this case $\sum_{k=1}^K (-1)^k a_k = A(\alpha) K +O(1)$ with some constant $A(\alpha)$. In fact, $A(\alpha)=0$ if $p$ is odd, and $A(\alpha)=p^{-1} \sum_{k=1}^p (-1)^{r+k} a_{r+k}$ (possibly zero) if $p$ is even. We also have $\log q_K=\Lambda(\alpha) K+O(1)$ with some constant $\Lambda(\alpha )>0$. In fact, $\Lambda(\alpha) = p^{-1} \log \eta$, where $\eta>1$ is the larger of the two eigenvalues of the matrix
\[ \left( \begin{array}{cc} 0 & 1 \\ 1 & a_{r+1} \end{array} \right) \left( \begin{array}{cc} 0 & 1 \\ 1 & a_{r+2} \end{array} \right) \cdots \left( \begin{array}{cc} 0 & 1 \\ 1 & a_{r+p} \end{array} \right) . \]
\begin{thm}\label{quadraticirrationaltheorem} Let $\alpha$ be a quadratic irrational, and let $A(\alpha)$ and $\Lambda (\alpha)$ be as above. There exists a constant $c(\alpha)>0$ such that
\[ D_2^2(S(\alpha, N)) = c(\alpha) \log N +O(1) , \]
and
\[ D_2^2(L(\alpha, N)) = \left\{ \begin{array}{ll} \frac{3}{2} c(\alpha) \log N + O((\log \log N)^4) & \textrm{if } A(\alpha)=0, \\ \frac{A(\alpha)^2}{144 \Lambda(\alpha )^2} \log^2 N + O(\log N) & \textrm{if } A(\alpha) \neq 0 . \end{array} \right. \]
The implied constants depend only on $\alpha$.
\end{thm}
\noindent We proved the same result for $S(\alpha, N)$ with the slightly worse error term $O(\log \log N)$ in a previous paper \cite{BO2}. In contrast to $A(\alpha)$ and $\Lambda(\alpha)$, there seems to be no simple way to compute the value of $c(\alpha)$ directly from the continued fraction expansion. The latter constant first appeared in certain lattice point counting problems studied in detail by Beck \cite{BE1,BE2,BE3}, who showed that it is related to the arithmetic of the ring of algebraic integers of the real quadratic field $\mathbb{Q}(\alpha)$, and computed its explicit value for any quadratic irrational; for instance,
\[ c \left( \frac{1+\sqrt{5}}{2} \right) = \frac{1}{30 \sqrt{5} \log \frac{1+\sqrt{5}}{2}} \quad \textrm{and} \quad c(\sqrt{3}) = \frac{1}{12 \sqrt{3} \log (2+\sqrt{3})} . \]

Precise results also follow for non-badly approximable irrationals whose continued fraction expansions are explicitly known. Consider Euler's number $e=[2;1,2,1,1,4,1,\dots, 1,2n,1,\dots ]$ as an illustration. Since the ``period length'' is odd, the square of the alternating sum $(\sum_{k=1}^K (-1)^k a_k)^2 \ll K^2$ is negligible compared to $\sum_{k=1}^K a_k^2 = (4/81)K^3+O(K^2)$. Thus from our general results it easily follows that
\[ D_2 (S(e,N)) = \frac{1}{3\sqrt{30}} \left( \frac{\log N}{\log \log N} \right)^{3/2} \left( 1 + O \left( \frac{\log \log \log N}{\log \log N} \right) \right) , \]
and
\[ D_2 (L(e,N)) = \frac{1}{6\sqrt{5}} \left( \frac{\log N}{\log \log N} \right)^{3/2} \left( 1 + O \left( \frac{\log \log \log N}{\log \log N} \right) \right) . \]
In contrast, e.g.\ for $\tan 1 = [1;1,1,3,1,5,1,\dots, 2n-1,1, \dots ]$, the ``period length'' is even, and the alternating sum $(\sum_{k=1}^K (-1)^k a_k)^2=K^4/16+O(K^3)$ dominates $\sum_{k=1}^K a_k^2 = K^3/6 + O(K^2)$. Consequently,
\[ D_2 (S(\tan 1,N)) = \frac{1}{3\sqrt{30}} \left( \frac{\log N}{\log \log N} \right)^{3/2} \left( 1 + O \left( \frac{\log \log \log N}{\log \log N} \right) \right) , \]
but for the unsymmetrized lattice we have the larger order of magnitude
\[ D_2 (L(\tan 1, N)) = \frac{1}{12} \left( \frac{\log N}{\log \log N} \right)^2 \left( 1 + O \left( \frac{\log \log \log N}{\log \log N} \right) \right) . \]

We also establish precise results for randomly chosen $\alpha$, starting with the asymptotics a.e.\ in the sense of the Lebesgue measure.
\begin{thm}\label{aeasymptotictheorem} Let $\varphi$ be a positive nondecreasing function on $(0,\infty)$.
\begin{enumerate}
\item[(i)] If $\sum_{n=1}^{\infty} 1/\varphi(n) < \infty$, then for a.e.\ $\alpha$,
\[ \begin{split} D_2(S(\alpha, N)) &\le \varphi (\log N) + O(\log N \log \log N), \\ D_2(L(\alpha, N)) &\le \varphi (\log N) + O(\log N \log \log N) \end{split} \]
with implied constants depending only on $\alpha$ and $\varphi$.
\item[(ii)] If $\sum_{n=1}^{\infty} 1/\varphi (n) = \infty$, then for a.e.\ $\alpha$,
\[ D_2 (S(\alpha, N)) \ge \varphi (\log N) \quad \textrm{and} \quad D_2 (L(\alpha, N)) \ge \varphi (\log N) \quad \textrm{for infinitely many } N. \]
\end{enumerate}
\end{thm}
\noindent In particular, for a.e.\ $\alpha$ we have $D_2 (S(\alpha, N)) \ll \log N (\log \log N)^{1+\varepsilon}$ and $D_2 (L(\alpha, N)) \ll \log N (\log \log N)^{1+\varepsilon}$ with any $\varepsilon >0$, but these fail with $\varepsilon =0$.

Our next result is the distributional analogue of Theorem \ref{aeasymptotictheorem}, stating that if $\alpha$ is chosen randomly from $[0,1]$ with an absolutely continuous distribution, then after suitable normalization $D_2^2(S(\alpha, N))$ converges to the standard L\'evy distribution. If $\alpha$ is chosen randomly with the Lebesgue measure $\lambda$ or the Gauss measure $\nu (B)=(1/\log 2) \int_B 1/(1+x) \, \mathrm{d}x$ ($B \subseteq [0,1]$ Borel) as distribution, then we also estimate the rate of convergence in the Kolmogorov metric.
\begin{thm}\label{irrationallimitdistributiontheorem} If $\mu$ is a Borel probability measure on $[0,1]$ which is absolutely continuous with respect to the Lebesgue measure, then for any $t \ge 0$,
\[ \mu \left( \left\{ \alpha \in [0,1] \, : \, 5 \pi^3 \frac{D_2^2 (S(\alpha, N))}{\log^2 N} \le t \right\} \right) \to \int_0^t \frac{e^{-1/(2x)}}{\sqrt{2 \pi} x^{3/2}} \, \mathrm{d} x \qquad \textrm{as } N \to \infty . \]
If $\mu$ is either the Lebesgue measure $\lambda$ or the Gauss measure $\nu$, then for any $N \ge 3$,
\[ \sup_{t \ge 0} \left| \mu \left( \left\{ \alpha \in [0,1] \, : \, 5 \pi^3 \frac{D_2^2 (S(\alpha, N))}{\log^2 N} \le t \right\} \right) - \int_0^t \frac{e^{-1/(2x)}}{\sqrt{2 \pi} x^{3/2}} \, \mathrm{d} x \right| \ll \frac{(\log \log N)^{1/3}}{(\log N)^{1/3}} \]
with a universal implied constant.
\end{thm}
\noindent We conjecture that a similar result holds for the unsymmetrized lattice as well, i.e.\ if $\alpha$ is chosen randomly from $[0,1]$ with an absolutely continuous distribution, then $D_2^2(L(\alpha, N))/\log^2 N$ has a nondegenerate limit distribution as $N \to \infty$.

Our results, especially Theorems \ref{optimalirrationaltheorem}, \ref{aeasymptotictheorem} and \ref{irrationallimitdistributiontheorem} should be compared to the corresponding properties of the discrepancy of the classical sequence $\{ n \alpha \}$, defined as
\[ \mathrm{Disc}_N(n \alpha ) = \sup_{[a,b] \subset [0,1)} \left| \sum_{n=1}^N I_{[a,b]}(\{ n \alpha \}) - N (b-a) \right| . \]
Here and for the rest of the paper, $I_S$ denotes the indicator function of a set $S$. Note that $\max_{1 \le \ell \le N} \mathrm{Disc}_{\ell}(n \alpha)$ is, up to a factor of $2$, equal to $D_{\infty}(L(\alpha, N))$, where the $L^{\infty}$ discrepancy (also called star-discrepancy) $D_{\infty}$ of a finite point set is defined as $D_2$ with the $L^2$ norm replaced by the $L^{\infty}$ norm. Roughly speaking, for $N \approx q_K$ we have $\max_{1 \le \ell \le N} \mathrm{Disc}_{\ell}( n \alpha ) \approx \sum_{k=1}^K a_k$. By a classical theorem of W.\ Schmidt \cite[p.\ 41]{DT}, the optimal rate for the discrepancy is $\log N$, and we can characterize all irrationals for which the optimum is attained \cite[p.\ 53]{DT} as
\[ \mathrm{Disc}_N (n \alpha ) \ll \log N \,\, \Longleftrightarrow \,\, \frac{1}{K} \sum_{k=1}^K a_k \ll 1 . \]
The discrepancy $\mathrm{Disc}_N(n \alpha )$ is also known to satisfy the same asymptotics a.e.\ as in Theorem \ref{aeasymptotictheorem} \cite[p.\ 63]{DT}. A fortiori, the previous two results apply also to $\max_{1 \le \ell \le N} \mathrm{Disc}_{\ell}(n \alpha)$, and hence to $D_{\infty}(L(\alpha, N))$. We mention two distributional analogues due to Kesten \cite{KE}:
\[ \begin{split} \frac{\mathrm{Disc}_N(n \alpha)}{\log N \log \log N} &\to \frac{2}{\pi^2} \quad \textrm{in measure,} \\ \frac{\max_{1 \le \ell \le N}\mathrm{Disc}_{\ell}(n \alpha)}{\log N \log \log N} &\to \frac{3}{\pi^2} \quad \textrm{in measure.} \end{split} \]

As a curious observation, we mention that there exists an irrational $\alpha$ such that
\[ \log N \ll D_2(S(\alpha, N)) \le D_{\infty} (S(\alpha, N)) \ll \log N , \]
and
\[ \log N \ll D_2(L(\alpha, N)) \le D_{\infty} (L(\alpha, N)) \ll \log N , \]
i.e.\ both $S(\alpha, N)$ and $L(\alpha, N)$ have optimal $L^{\infty}$ discrepancy, but neither has optimal $L^2$ discrepancy. Indeed, it is easy to construct\footnote{E.g.\ let $a_k=k$ if $k$ is a power of $2$, and $a_k=1$ otherwise.} a sequence of positive integers $a_k$ such that $K^{-1} \sum_{k=1}^K a_k \ll 1$ but $\sum_{k=1}^K a_k^2 \gg K^2$.

Consider now the case of a rational $\alpha$. For the sake of simplicity, we will always assume that $N$ is the denominator of $\alpha$. That is, given a reduced fraction $p/q$, we study the $q$-element set
\[ L(p/q,q) = \left\{ \left( \left\{ \frac{np}{q} \right\}, \frac{n}{q} \right) \in [0,1)^2 \, : \, 0 \le n \le q-1 \right\} , \]
and the $2q$-element set
\[ S(p/q,q) = \left\{ \left( \left\{ \pm \frac{np}{q} \right\}, \frac{n}{q} \right) \in [0,1)^2 \, : \, 0 \le n \le q-1 \right\} . \]
The characterization of all rationals for which the $L^2$ discrepancy is optimal is exactly the same as in the irrational case.
\begin{thm}\label{optimalrationaltheorem} Let $p/q=[a_0;a_1,\dots, a_r]$ be a reduced rational. We have
\[ \begin{split} D_2(S(p/q,q)) \ll \sqrt{\log q} \,\, &\Longleftrightarrow \,\, \frac{1}{r} \sum_{k=1}^r a_k^2 \ll 1, \\ D_2(L(p/q,q)) \ll \sqrt{\log q} \,\, &\Longleftrightarrow \,\, \frac{1}{r} \sum_{k=1}^r a_k^2 \ll 1 \textrm{ and } \frac{1}{\sqrt{r}} \left| \sum_{k=1}^r (-1)^k a_k \right| \ll 1. \end{split} \]
\end{thm}

As an analogue of the metric results on typical values of $\alpha$ in the sense of the Lebesgue measure above, we also study the $L^2$ discrepancy for typical values of rationals. In this case, ``typical'' means choosing a reduced fraction $p/q$ randomly from the set of all reduced rationals with bounded denominator.
\begin{thm}\label{rationallimitdistributiontheorem} Let $F_Q$ denote the set of all reduced fractions in $[0,1]$ with denominator at most $Q$. For any $Q \ge 2$,
\[ \sup_{t \ge 0} \left| \frac{1}{|F_Q|} \left| \left\{ \frac{p}{q} \in F_Q \, : \, 5 \pi^3 \frac{D_2^2 (S(p/q,q))}{\log^2 q} \le t \right\} \right| - \int_0^t \frac{e^{-1/(2x)}}{\sqrt{2 \pi} x^{3/2}} \, \mathrm{d} x \right| \ll \frac{1}{(\log Q)^{1/2}} \]
with a universal implied constant.
\end{thm}
\noindent We conjecture that a similar result holds for the unsymmetrized lattice as well, i.e.\ if $p/q$ is chosen randomly from $F_Q$, then $D_2^2(L(p/q,q))/\log^2 q$ has a nondegenerate limit distribution as $Q \to \infty$.

In Section \ref{parsevalsection}, we derive an explicit formula for $D_2(S(\alpha, N))$ and $D_2(L(\alpha, N))$ in terms of the partial quotients of $\alpha$, see Propositions \ref{parsevalprop} and \ref{simpleparsevalprop}. Theorems \ref{optimalirrationaltheorem}, \ref{quadraticirrationaltheorem} and \ref{optimalrationaltheorem} are proved in Section \ref{optimalsubsection}. In Section \ref{typicalirrationalsection}, we show how Theorems \ref{aeasymptotictheorem} and \ref{irrationallimitdistributiontheorem} follow from classical results on the metric theory of continued fractions and $\psi$-mixing random variables. The proof of Theorem \ref{rationallimitdistributiontheorem} in Section \ref{typicalrationalsection}, on the other hand, relies on recent results of Bettin and Drappeau \cite{BD} on the statistics of partial quotients of random rationals.

\section{$L^2$ discrepancy via the Parseval formula}\label{parsevalsection}

\subsection{The main estimates}

We remind that $\alpha = [a_0;a_1,a_2, \dots]$ is the (finite or infinite) continued fraction expansion of a real number $\alpha$, and $p_k/q_k=[a_0;a_1,\dots, a_k]$ denotes its convergents. For the rest of the paper, we also use the notation
\[ T_n=\sum_{\ell=0}^n \left( \frac{1}{2} - \{ \ell \alpha \} \right) \quad \textrm{and} \quad E_N=\frac{1}{N}\sum_{n=0}^{N-1} T_n . \]
For the sake of readability, $a=b \pm c$ denotes $|a-b| \le c$, and $\zeta$ is the Riemann zeta function.

Our main tool is an evaluation of the $L^2$ discrepancy up to a small error, based on the Parseval formula. This method goes back to Davenport \cite{DA}, and more recently has also been used in \cite{BI,BTY1,BTY2,HKP,RS}. We follow the steps in our previous paper \cite{BO2}, where we considered irrationals whose sequence of partial quotients is reasonably well-behaved (e.g.\ bounded, or increasing at a regular rate such as for Euler's number). Here we shall need a more refined analysis in order to study arbitrary reals without any assumption on the partial quotients.
\begin{prop}\label{parsevalprop} For any $q_{K-1} \le N \le q_K$, we have
\[ D_2^2 (S(\alpha, N)) = \sum_{m=1}^{q_{K-1}-1} \frac{1}{4 \pi^4 m^2 \| m \alpha \|^2} + \xi_S(\alpha, N) \pm \left( \sum_{k=0}^{K-1} \frac{a_{k+1}}{2q_k} + \frac{\zeta(3)}{16 \pi^4 N} \sum_{k=0}^{K-2} (a_{k+1}+2)^3 q_k + 6.28  \right) \]
with some $\xi_S(\alpha, N)$ which satisfies both $0 \le \xi_S(\alpha, N) \le \sum_{m=q_{K-1}}^{q_K-1} \frac{1}{2 \pi^4 m^2 \| m \alpha \|^2}$ and
\[ \xi_S(\alpha, N) = \sum_{m=q_{K-1}}^{q_K-1} \frac{1}{4 \pi^4 m^2 \| m \alpha \|^2} \pm \left( \frac{\zeta(3)}{16 \pi^4 N} (a_K+2)^3 q_{K-1} +0.07 \right) .  \]
Similarly, for any $q_{K-1} \le N \le q_K$, we have
\[ \begin{split} D_2^2(L(\alpha, N)) = &\frac{1}{N} \sum_{n=0}^{N-1} \left( T_n^2 + \frac{1}{2}T_n \right) + \left( 1-\frac{1}{2N} \right) \sum_{m=1}^{q_{K-1}-1} \frac{1}{4 \pi^4 m^2 \| m \alpha \|^2}  \\ &+\xi_L(\alpha, N) \pm \left( \sum_{k=0}^{K-1} \frac{a_{k+1}}{8 q_k} + \frac{\zeta(3)}{16 \pi^4 N} \sum_{k=0}^{K-2} (a_{k+1}+2)^3 q_k + 2.78 \right) \end{split} \]
with some $\xi_L(\alpha, N)$ which satisfies both $0 \le \xi_L(\alpha, N) \le \sum_{m=q_{K-1}}^{q_K-1} \frac{1}{2 \pi^4 m^2 \| m \alpha \|^2}$ and
\[ \xi_L(\alpha, N) = \left( 1-\frac{1}{2N} \right) \sum_{m=q_{K-1}}^{q_K-1} \frac{1}{4 \pi^4 m^2 \| m \alpha \|^2} \pm \frac{\zeta(3)}{16 \pi^4 N} (a_K+2)^3 q_{K-1} . \]
\end{prop}
\noindent We also prove a simpler form which is sharp up to a constant factor.
\begin{prop}\label{simpleparsevalprop} For any $q_{K-1} \le N \le q_K$, we have $D_2^2 (S(\alpha, N)) \ll \sum_{k=1}^K a_k^2$. For $N=q_K$, we also have $D_2^2 (S(\alpha, q_K)) \gg \sum_{k=1}^K a_k^2$, and
\[ \sum_{k=1}^K a_k^2 + \left( \sum_{k=1}^K (-1)^k a_k \right)^2 \ll D_2^2 (L(\alpha, q_K)) \ll \sum_{k=1}^K a_k^2 + \left( \sum_{k=1}^K (-1)^k a_k \right)^2 . \]
The implied constants are universal.
\end{prop}
\noindent We postpone the proofs to Sections \ref{section2.3} and \ref{section2.4}, and now comment on the main terms.

The contribution of the sums $T_n$ can be written as
\[ \frac{1}{N} \sum_{n=0}^{N-1} \left( T_n^2 + \frac{1}{2} T_n \right) = \frac{1}{N} \sum_{n=0}^{N-1} (T_n-E_N)^2 + E_N^2 + \frac{1}{2} E_N . \]
Observing a connection with Dedekind sums, Beck showed \cite[p.\ 79 and p.\ 91]{BE1} (see also \cite{SCH}) that for any $q_{K-1} \le N \le q_K$, the ``expected value'' $E_N$ is
\begin{equation}\label{EN}
E_N = \frac{1}{12} \sum_{k=1}^K (-1)^k a_k +O \left( \max_{1 \le k \le K} a_k \right) .
\end{equation}
For $N=q_K$, the error term can be improved to
\begin{equation}\label{EqK}
E_{q_K} = \frac{1}{12} \sum_{k=1}^K (-1)^k a_k +O(1) .
\end{equation}
Both implied constants are universal. Generalizing results of Beck, in a recent paper \cite{BO1} we proved that if $a_k \le c k^d$ with some constants $c>0$ and $d \ge 0$, then for any $q_{K-1} \le N \le q_K$, the ``variance'' is
\begin{equation}\label{Tnvariance}
\frac{1}{N} \sum_{n=0}^{N-1} (T_n-E_N)^2 = \sum_{m=1}^{q_K-1} \frac{1}{8 \pi^4 m^2 \| m \alpha \|^2} +O \left( \max_{|k-K| \ll \log K} a_k^2 \cdot (\log \log N)^4 \right)
\end{equation}
with implied constants depending only on $c$ and $d$. See also Lemma \ref{Tnlemma} below.

Finally, we will need two different evaluations of the Diophantine sum appearing in Proposition \ref{parsevalprop}. On the one hand, for general $\alpha$ we have \cite[p.\ 110]{BO3}, \cite{BO2}
\begin{equation}\label{diophantineevaluation}
\sum_{m=1}^{q_K-1} \frac{1}{m^2 \| m \alpha \|^2} = \frac{\pi^4}{90} \sum_{k=1}^K a_k^2 \pm 152 \sum_{k=1}^K a_k .
\end{equation}
On the other hand, Beck \cite[p.\ 176]{BE1} proved that if $\alpha$ is quadratic irrational, then for any $M \ge 1$,
\begin{equation}\label{diophantinesumbeck}
\sum_{m=1}^M \frac{1}{4 \pi^4 m^2 \| m \alpha \|^2} = c(\alpha) \log M +O(1)
\end{equation}
with some constant $c(\alpha)>0$ and an implied constant depending only on $\alpha$.

\subsection{Optimal lattices}\label{optimalsubsection}

In this section, we deduce Theorems \ref{optimalirrationaltheorem}, \ref{quadraticirrationaltheorem} and \ref{optimalrationaltheorem} from Propositions \ref{parsevalprop} and \ref{simpleparsevalprop}.

\begin{proof}[Proof of Theorem \ref{optimalirrationaltheorem}] Consider first the symmetrized lattice $S(\alpha, N)$. We will show the implications
\[ \frac{1}{K} \sum_{k=1}^K a_k^2 \ll 1 \,\, \Longrightarrow \,\, D_2(S(\alpha, N)) \ll \sqrt{\log N} \,\, \Longrightarrow \,\, D_2(S(\alpha, q_K)) \ll \sqrt{\log q_K} \,\, \Longrightarrow \,\, \frac{1}{K} \sum_{k=1}^K a_k^2 \ll 1 . \]
Assume that $K^{-1} \sum_{k=1}^K a_k^2 \ll 1$ as $K \to \infty$. By Proposition \ref{simpleparsevalprop}, for any $q_{K-1} \le N \le q_K$ we have $D_2^2 (S(\alpha, N)) \ll \sum_{k=1}^K a_k^2 \ll K \ll \log N$, as claimed. The second implication is trivial. Next, assume that $D_2 (S(\alpha, N)) \ll \sqrt{\log N}$ as $N \to \infty$. By Proposition \ref{simpleparsevalprop}, for $N=q_K$ we have
\[ \sum_{k=1}^K a_k^2 \ll D_2^2 (S(\alpha, q_K)) \ll \log q_K \le \sum_{k=1}^K \log (a_k+1) \ll \sum_{k=1}^K a_k \le \sqrt{K \sum_{k=1}^K a_k^2}, \]
and the claim follows. This finishes the proof of the equivalence for $S(\alpha, N)$.

Consider now the unsymmetrized lattice $L(\alpha, q_K)$. Assume that $K^{-1} \sum_{k=1}^K a_k^2 \ll 1$ and $K^{-1/2} \left| \sum_{k=1}^K (-1)^k a_k \right| \ll 1$ as $K \to \infty$. By Proposition \ref{simpleparsevalprop}, for $N=q_K$ we have
\[ D_2^2 (L(\alpha, q_K)) \ll \sum_{k=1}^K a_k^2 + \left( \sum_{k=1}^K (-1)^k a_k \right)^2 \ll K \ll \log q_K, \]
as claimed. Next, assume that $D_2(L(\alpha, q_K)) \ll \sqrt{\log q_K}$ as $K \to \infty$. By Proposition \ref{simpleparsevalprop}, for $N=q_K$ we have
\[ \sum_{k=1}^K a_k^2 + \left( \sum_{k=1}^K (-1)^k a_k \right)^2 \ll D_2^2(L(\alpha, q_K)) \ll \log q_K . \]
Hence both $\sum_{k=1}^K a_k^2 \ll \log q_K$ and  $\left( \sum_{k=1}^K (-1)^k a_k \right)^2 \ll \log q_K$. As above, the former estimate shows that $K^{-1} \sum_{k=1}^K a_k^2 \ll 1$. In particular, $\log q_K \le \sum_{k=1}^K \log (a_k+1) \ll \sum_{k=1}^K a_k^2 \ll K$, therefore $\left( \sum_{k=1}^K (-1)^k a_k \right)^2 \ll K$, as claimed. This finishes the proof of the equivalence for $L(\alpha, q_K)$.
\end{proof}

\begin{proof}[Proof of Theorem \ref{optimalrationaltheorem}] As Proposition \ref{simpleparsevalprop} applies to both rationals and irrationals, the proof is identical to that of Theorem \ref{optimalirrationaltheorem}.
\end{proof}

\begin{proof}[Proof of Theorem \ref{quadraticirrationaltheorem}] Let $\alpha$ be a quadratic irrational. By Proposition \ref{parsevalprop} and formula \eqref{diophantinesumbeck}, for any $q_{K-1} \le N \le q_K$,
\[ D_2^2 (S(\alpha, N)) = \sum_{m=1}^{q_K-1} \frac{1}{4 \pi^4 m^2 \| m \alpha \|^2} +O(1)=c(\alpha) \log N +O(1) , \]
as claimed. Using also formula \eqref{Tnvariance}, we similarly get
\[ D_2^2 (L(\alpha, N)) = \frac{3}{2} c(\alpha) \log N + E_N^2+\frac{1}{2} E_N +O((\log \log N)^4). \]
Formula \eqref{EN} shows that here $E_N= \frac{A(\alpha)}{12} K+O(1)=\frac{A(\alpha)}{12 \Lambda(\alpha)} \log N +O(1)$, and the claim follows.
\end{proof}

\subsection{Proof of Proposition \ref{parsevalprop}}\label{section2.3}

\begin{lem}\label{diophantinelemma}\hspace{1mm}
\begin{enumerate}
\item[(i)] For any $K \ge 1$,
\[ \sum_{m=1}^{q_K-1} \frac{1}{\pi^2 m^2 \| m \alpha \|} \le \sum_{k=0}^{K-1} \frac{a_{k+1}}{2 q_k} + 3.12. \]
\item[(ii)] For any $K \ge 1$ and $n \ge 0$,
\[ \sum_{m=q_K}^{\infty} \frac{1}{2 \pi^2 m^2} \min \left\{ \frac{1}{4 \| m \alpha \|^2}, n^2 \right\} \le 1.12 \frac{n}{q_K} + 0.61 \frac{n^2}{q_K^2}. \]
\item[(iii)] For any $K \ge 1$ and $N \ge q_{K-1}$,
\[ \sum_{m=1}^{q_K-1} \frac{1}{4 \pi^4 m^2 \| m \alpha \|^2} \min \left\{ \frac{1}{4 N \| 2 m \alpha \|}, 1 \right\} \le \frac{\zeta(3)}{16 \pi^4 N} \sum_{k=0}^{K-1} (a_{k+1}+2)^3 q_k+ 0.07. \]
\end{enumerate}
\end{lem}

\begin{proof} The proof of all three claims is based on the following simple observations. Let $k \ge 1$, or $k=0$ and $a_1>1$. For any integer $a \ge 1$ let $J_{k,a}=[aq_k, (a+1)q_k) \cap [q_k,q_{k+1})$ be a (possibly empty) index set. Let $\delta_k=q_k \alpha -p_k$, and recall from the general theory of continued fractions that $1/(q_{k+1}+q_k) \le |\delta_k|=\| q_k \alpha \| \le 1/q_{k+1}$. For any integer $m \in J_{k,a}$, we have $m \alpha = mp_k/q_k + m \delta_k/q_k$, and here the second term is negligible as $m|\delta_k|/q_k<1/q_k$. Since $p_k$ and $q_k$ are relatively prime, as $m$ runs in the index set $J_{k,a}$, the numbers $mp_k$ attain each mod $q_k$ residue class at most once. If $mp_k \not\equiv 0, \pm 1 \pmod{q_k}$, then
\[ \| m \alpha \| = \left\| \frac{mp_k}{q_k} + \frac{m \delta_k}{q_k} \right\| \ge \left\| \frac{mp_k}{q_k} \right\| - \frac{1}{q_k} \ge \frac{1}{2} \left\| \frac{mp_k}{q_k} \right\| . \]
Therefore for any nondecreasing function $f: [2,\infty ) \to [0,\infty )$, we have
\begin{equation}\label{fdiophantinesum}
\sum_{m \in J_{k,a}} f \left( \frac{1}{\| m \alpha \|} \right) \le 3 f \left( \frac{1}{\| q_k \alpha \|} \right) + \sum_{j=2}^{q_k-2} f \left( \frac{2}{\| j/q_k \|} \right) \le 3f \left( \frac{1}{\| q_k \alpha \|} \right) + 2 \sum_{2 \le j \le q_k/2} f \left( \frac{2q_k}{j} \right) .
\end{equation}
Note that $3 f(1/\| q_k \alpha \| )$ is an upper bound to the contribution of the three terms for which $m p_k \equiv 0, \pm 1 \pmod{q_k}$.

We also have the simpler estimate
\begin{equation}\label{fdiophantinesumsimple}
\sum_{1 \le m < q_{k+1}} f \left( \frac{1}{\| m \alpha \|} \right) \le 2 \sum_{1 \le j \le q_{k+1}/2} f \left( \frac{1}{j \| q_k \alpha \|} \right) .
\end{equation}
Indeed, consider the points $m \alpha \pmod{1}$, $1 \le m < q_{k+1}$ and the intervals $H_j=[j \| q_k \alpha \|, (j+1)\| q_k \alpha \|)$, $j \ge 1$ and $H_j = ((j-1) \| q \alpha \|, j \| q_k \alpha \| ]$, $j \le -1$. Since $\| (m_1-m_2) \alpha \| \ge \| q_k \alpha \|$ for any $m_1, m_2 \in [1,q_{k+1})$, $m_1 \neq m_2$, each interval $H_j$ contains at most one point $m \alpha \pmod{1}$, and \eqref{fdiophantinesumsimple} follows.

\noindent\textbf{(i)} Estimate \eqref{fdiophantinesum} yields
\[ \sum_{m \in J_{k,a}} \frac{1}{\pi^2 m^2 \| m \alpha \|} \le \frac{1}{\pi^2 a^2 q_k^2} \left( \frac{3}{\| q_k \alpha \|} + 2 \sum_{2 \le j \le q_k/2} \frac{2q_k}{j} \right) \le \frac{1}{\pi^2 a^2 q_k^2} \left( 3(q_{k+1}+q_k) + 4 q_k \log \frac{q_k}{2} \right) . \]
Summing over $a \ge 1$ and\footnote{If $a_1=1$, then the term $k=0$ can be removed.} $0 \le k \le K-1$ leads to
\[ \begin{split} \sum_{m=1}^{q_K-1} \frac{1}{\pi^2 m^2 \| m \alpha \|} &\le \sum_{k=0}^{K-1} \frac{3 q_{k+1} + 3q_k + 4 q_k \log (q_k/2)}{6 q_k^2} \\ &\le \sum_{k=0}^{K-1} \frac{a_{k+1}}{2 q_k} + \sum_{k=0}^{K-1} \frac{3+2\log (q_k/2)}{3 q_k} \\ &\le \sum_{k=0}^{K-1} \frac{a_{k+1}}{2 q_k} + \sum_{k=0}^{\infty} \frac{3+2 \log (F_{k+1}/2)}{3 F_{k+1}} , \end{split} \]
where $F_{k+1}$ are the Fibonacci numbers. The numerical value of the series in the previous line is $3.1195\dots$, as claimed.

\noindent\textbf{(ii)} Estimate \eqref{fdiophantinesum} yields
\[ \sum_{m \in J_{k,a}} \frac{1}{2 \pi^2 m^2} \min \left\{ \frac{1}{4 \| m \alpha \|^2}, n^2 \right\} \le \frac{1}{2 \pi^2 a^2 q_k^2} \left( 3n^2 + 2 \sum_{j=2}^{\infty} \min \left\{ \frac{q_k^2}{j^2} ,n^2 \right\} \right) \le \frac{1}{2 \pi^2 a^2 q_k^2} \left( 3n^2 + 4 n q_k \right) . \]
Note that the contribution of the terms $2 \le j \le \lfloor q_k/n \rfloor +1$ and $j \ge \lfloor q_k/n \rfloor +2$ is at most $n q_k$ each. Summing over $a \ge 1$ and $k \ge K$ leads to
\[ \sum_{m=q_K}^{\infty} \frac{1}{2 \pi^2 m^2} \min \left\{ \frac{1}{4 \| m \alpha \|^2}, n^2 \right\} \le \sum_{k=K}^{\infty} \frac{3n^2 + 4nq_k}{12 q_k^2} . \]
From the recursion satisfied by $q_k$ one readily sees that $q_{K+\ell} \ge F_{\ell+1} q_K$ for all $\ell \ge 0$, hence the right hand side of the previous formula is at most $c_1 n/q_K + c_2n^2/q_K^2$ with $c_1=\sum_{\ell =0}^{\infty} 1/(3F_{\ell +1}) = 1.1199\dots$ and $c_2=\sum_{\ell =0}^{\infty} 1/(4F_{\ell +1}^2)=0.6065\dots$, as claimed.

\noindent\textbf{(iii)} The contribution of all $m$ such that $\| m \alpha \| > 1/4$ is negligible:
\[ \sum_{\substack{1 \le m \le q_K-1 \\ \| m \alpha \| > 1/4}} \frac{1}{4 \pi^4 m^2 \| m \alpha \|^2} \min \left\{ \frac{1}{4 N \| 2 m \alpha \|}, 1 \right\} < \sum_{m=1}^{\infty} \frac{4}{\pi^4 m^2} = \frac{2}{3 \pi^2} . \]
On the other hand, $\| m \alpha \| \le 1/4$ implies $\| 2m \alpha \| = 2 \| m \alpha \|$, hence the contribution of all such terms is
\[ \sum_{\substack{1 \le m \le q_K-1 \\ \| m \alpha \| \le 1/4}} \frac{1}{4 \pi^4 m^2 \| m \alpha \|^2} \min \left\{ \frac{1}{4 N \| 2 m \alpha \|}, 1 \right\} \le \sum_{m=1}^{q_K-1} \frac{1}{32 \pi^4 N m^2 \| m \alpha \|^3} . \]
Estimate \eqref{fdiophantinesumsimple} gives
\[ \sum_{q_k \le m < q_{k+1}} \frac{1}{32 \pi^4 N m^2 \| m \alpha \|^3} \le \frac{1}{16 \pi^4 N q_k^2} \sum_{j=1}^{\infty} \frac{1}{j^3 \| q_k \alpha \|^3} \le \frac{\zeta(3) (a_{k+1}+2)^3 q_k}{16 \pi^4 N} . \]
Summing over $0 \le k \le K-1$, we thus obtain
\[ \sum_{m=1}^{q_K-1} \frac{1}{4 \pi^4 m^2 \| m \alpha \|^2} \min \left\{ \frac{1}{4 N \| 2 m \alpha \|}, 1 \right\} \le \sum_{k=0}^{K-1} \frac{\zeta(3) (a_{k+1}+2)^3 q_k}{16 \pi^4 N} + \frac{2}{3 \pi^2} . \]
Here $2/(3 \pi^2)=0.06754\dots$, as claimed.
\end{proof}

\begin{proof}[Proof of Proposition \ref{parsevalprop}] We give a detailed proof for the symmetrized lattice $S(\alpha, N)$, and then indicate at the end how to modify the proof for the unsymmetrized lattice $L(\alpha, N)$.

Let $B(x,y)=|S(\alpha, N) \cap ([0,x) \times [0,y))|$ denote the number of points of $S(\alpha, N)$ which fall into the box $[0,x) \times [0,y)$. Integrating on the strips $[0,1) \times [n/N,(n+1)/N)$ separately leads to
\[ D_2^2 (S(\alpha, N)) = \sum_{n=0}^{N-1} \int_0^1 \int_{\frac{n}{N}}^{\frac{n+1}{N}} \left( B(x,y) - 2N xy \right)^2 \, \mathrm{d}y \, \mathrm{d}x = M+R+\frac{4}{9} \]
with
\[ \begin{split} M&:=\frac{1}{N} \sum_{n=0}^{N-1} \int_0^1 \left( B \left( x, \frac{n+1}{N} \right) - 2(n+1)x \right)^2 \, \mathrm{d}x, \\ R&:= \frac{2}{N} \sum_{n=0}^{N-1} \int_0^1 \left( B \left( x, \frac{n+1}{N} \right) - 2(n+1)x \right) x \, \mathrm{d}x . \end{split} \]
The function
\[ B \left( x, \frac{n+1}{N} \right) -2(n+1)x = \sum_{\ell =0}^{n} \left( I_{[0,x)}(\{ \ell \alpha \}) + I_{[0,x)}(\{ -\ell \alpha \}) -2x \right) \]
is mean zero, and has Fourier coefficients
\[ \begin{split} \int_0^1 \left( B \left( x, \frac{n+1}{N} \right) -2(n+1)x \right) e^{-2 \pi i m x} \, \mathrm{d}x &= \sum_{\ell =0}^n \frac{\cos (2 \ell m \pi \alpha)}{\pi i m} \\ &= \frac{1}{2 \pi i m} \left( \frac{\sin ((2n+1)m \pi \alpha )}{\sin (m \pi \alpha )} + 1 \right) . \end{split} \]
The Fourier coefficients of $x$ are $\int_0^1 x e^{-2\pi i m x} \, \mathrm{d}x=-1/(2 \pi i m)$, thus by the Parseval formula we have
\[ R = \frac{2}{N} \sum_{n=0}^{N-1} 2 \sum_{m=1}^{\infty} \frac{1}{2 \pi i m} \left( \frac{\sin ((2n+1)m \pi \alpha )}{\sin (m \pi \alpha)} + 1 \right) \cdot \frac{-1}{2 \pi i m} = \frac{1}{N} \sum_{n=0}^{N-1} \sum_{m=1}^{\infty} \frac{\sin ((2n+1) m \pi \alpha )}{\pi^2 m^2 \sin (m \pi \alpha)} +\frac{1}{6} . \]
The Parseval formula similarly gives
\[ \begin{split} M &= \frac{1}{N} \sum_{n=0}^{N-1} 2 \sum_{m=1}^{\infty} \frac{1}{4 \pi^2 m^2} \left( \frac{(\sin((2n+1)m \pi \alpha)}{\sin (m \pi \alpha)} + 1 \right)^2 \\ &= \frac{1}{N} \sum_{n=0}^{N-1} \sum_{m=1}^{\infty} \frac{\sin^2((2n+1)m \pi \alpha)}{2\pi^2 m^2 \sin^2 (m \pi \alpha)} + \frac{1}{N} \sum_{n=0}^{N-1} \sum_{m=1}^{\infty} \frac{\sin ((2n+1)m \pi \alpha )}{\pi^2 m^2 \sin (m \pi \alpha )} + \frac{1}{12} . \end{split} \]
We can estimate the total error in the previous two formulas using
\[ \left| \frac{\sin ((2n+1)m \pi \alpha )}{\sin (m \pi \alpha)} \right| \le \min \left\{ \frac{1}{2 \| m \alpha \|}, 2n+1 \right\} \]
and Lemma \ref{diophantinelemma} (i) as
\[ \begin{split} \bigg| \frac{1}{N} \sum_{n=0}^{N-1} \sum_{m=1}^{\infty} \frac{2 \sin ((2n+1)m \pi \alpha )}{\pi^2 m^2 \sin (m \pi \alpha )} \bigg| &\le \frac{1}{N} \sum_{n=0}^{N-1} \left( \sum_{m=1}^{q_K-1} \frac{1}{\pi^2 m^2 \| m \alpha \|} + \sum_{m=q_K}^{\infty} \frac{2(2n+1)}{\pi^2 m^2} \right) \\ &\le \sum_{k=0}^{K-1} \frac{a_{k+1}}{2q_k} +3.12+ \frac{4N}{\pi^2 q_K} . \end{split} \]
By the assumption $N \le q_K$ and the fact $3.12+4/\pi^2+4/9+1/6+1/12 <4.22$, we thus obtain
\[ D_2^2 (S(\alpha, N)) = \frac{1}{N} \sum_{n=0}^{N-1} \sum_{m=1}^{\infty} \frac{\sin^2((2n+1)m \pi \alpha)}{2\pi^2 m^2 \sin^2 (m \pi \alpha)} \pm \left( \sum_{k=0}^{K-1} \frac{a_{k+1}}{2q_k} + 4.22 \right) . \]
Lemma \ref{diophantinelemma} (ii) estimates the tail of the infinite series in the previous formula as
\[ \sum_{m=q_K}^{\infty}  \frac{\sin^2((2n+1)m \pi \alpha)}{2\pi^2 m^2 \sin^2 (m \pi \alpha)} \le \sum_{m=q_K}^{\infty} \frac{1}{2 \pi^2 m^2} \min \left\{ \frac{1}{4 \| m \alpha \|^2}, (2n+1)^2 \right\} \le 1.12 \frac{2n+1}{q_K} + 0.61 \frac{(2n+1)^2}{q_K^2} . \]
By the assumption $N \le q_K$ and the facts $\sum_{n=0}^{N-1}(2n+1)^2 \le (4/3)N^3$ and $4.22+1.12+(4/3)\cdot 0.61<6.16$, we immediately get
\[ D_2^2 (S(\alpha, N)) = \frac{1}{N} \sum_{n=0}^{N-1} \sum_{m=1}^{q_K-1} \frac{\sin^2((2n+1)m \pi \alpha)}{2\pi^2 m^2 \sin^2 (m \pi \alpha)} \pm \left( \sum_{k=0}^{K-1} \frac{a_{k+1}}{2q_k} + 6.16 \right) . \]
Elementary calculations show that the function $1/\sin^2 (\pi x) - 1/(\pi^2 \| x \|^2)$ is increasing on $(0,1/2]$, hence $1/(\pi^2 \| x \|^2) \le 1/\sin^2 (\pi x) \le 1/(\pi^2 \| x \|^2) + 1-4/\pi^2$ for all $x$. The error of replacing $\sin^2 (m \pi \alpha)$ by $\pi^2 \| m \alpha \|^2$ in the denominator of the previous formula is thus at most
\[ \frac{1}{N} \sum_{n=0}^{N-1} \sum_{m=1}^{q_K-1} \frac{\sin^2 ((2n+1)m \pi \alpha) (1-4/\pi^2)}{2 \pi^2 m^2} \le \sum_{m=1}^{\infty} \frac{1-4/\pi^2}{2 \pi^2 m^2} = \frac{1-4/\pi^2}{12} . \]
Since $6.16+(1-4/\pi^2)/12 <6.21$, we obtain
\begin{equation}\label{dsquareformula}
D_2^2 (S(\alpha, N)) = \frac{1}{N} \sum_{n=0}^{N-1} \sum_{m=1}^{q_{K-1}-1} \frac{\sin^2((2n+1)m \pi \alpha)}{2\pi^4 m^2 \| m \alpha \|^2} + \xi_S(\alpha, N) \pm \left( \sum_{k=0}^{K-1} \frac{a_{k+1}}{2q_k} + 6.21 \right) ,
\end{equation}
where we define
\[ \xi_S(\alpha, N) := \frac{1}{N} \sum_{n=0}^{N-1} \sum_{m=q_{K-1}}^{q_K-1} \frac{\sin^2((2n+1)m \pi \alpha)}{2\pi^4 m^2 \| m \alpha \|^2} . \]
Using the trigonometric identity
\[ \frac{1}{N} \sum_{n=0}^{N-1} \sin^2 ((2n+1)x) = \frac{1}{2} - \frac{\sin(4Nx)}{4N\sin(2x)}, \]
the first term in \eqref{dsquareformula} simplifies to
\[ \sum_{m=1}^{q_{K-1}-1} \frac{1}{4 \pi^4 m^2 \| m \alpha \|^2} - \sum_{m=1}^{q_{K-1}-1} \frac{\sin (4Nm \pi \alpha)}{8 \pi^4 N m^2 \| m \alpha \|^2 \sin(2m\pi \alpha )} . \]
Here second term can be estimated using Lemma \ref{diophantinelemma} (iii) as
\[ \begin{split} \left| \sum_{m=1}^{q_{K-1}-1} \frac{\sin (4Nm \pi \alpha)}{8 \pi^4 N m^2 \| m \alpha \|^2 \sin(2m\pi \alpha )} \right| &\le \sum_{m=1}^{q_{K-1}-1} \frac{1}{4 \pi^4 m^2 \| m \alpha \|^2} \min \left\{ \frac{1}{4N \| 2m \alpha \|} , 1 \right\} \\ &\le \frac{\zeta(3)}{16 \pi^4 N} \sum_{k=0}^{K-2} (a_{k+1}+2)^3 q_k + 0.07. \end{split} \]
Therefore \eqref{dsquareformula} simplifies to
\[ D_2^2 (S(\alpha, N)) = \sum_{m=1}^{q_{K-1}-1} \frac{1}{4 \pi^4 m^2 \| m \alpha \|^2} + \xi_S(\alpha, N) \pm \left( \sum_{k=0}^{K-1} \frac{a_{k+1}}{2q_k} + \frac{\zeta(3)}{16 \pi^4 N} \sum_{k=0}^{K-2} (a_{k+1}+2)^3 q_k + 6.28  \right) , \]
and it remains to prove the properties of $\xi_S(\alpha, N)$. Clearly, $0 \le \xi_S(\alpha, N) \le \sum_{m=q_{K-1}}^{q_K-1} \frac{1}{2 \pi^4 m^2 \| m \alpha \|^2}$. On the other hand, repeating arguments from above and from Lemma \ref{diophantinelemma} (iii), we can also write
\[ \begin{split} \xi_S(\alpha, N) &= \sum_{m=q_{K-1}}^{q_K-1} \frac{1}{4 \pi^4 m^2 \| m \alpha \|^2} - \sum_{m=q_{K-1}}^{q_K-1} \frac{\sin (4Nm \pi \alpha)}{8 \pi^4 N m^2 \| m \alpha \|^2 \sin(2m\pi \alpha )} \\ &= \sum_{m=q_{K-1}}^{q_K-1} \frac{1}{4 \pi^4 m^2 \| m \alpha \|^2} \pm \sum_{m=q_{K-1}}^{q_K-1} \frac{1}{4 \pi^4 m^2 \| m \alpha \|^2} \min \left\{ \frac{1}{4N \| 2m \alpha \|}, 1 \right\} \\ &= \sum_{m=q_{K-1}}^{q_K-1} \frac{1}{4 \pi^4 m^2 \| m \alpha \|^2} \pm \left( \frac{\zeta(3)}{16 \pi^4 N} (a_K+2)^3 q_{K-1} +0.07 \right) . \end{split} \]
This finishes the proof for $S(\alpha, N)$.

The proof for $L(\alpha, N)$ is entirely analogous. The only difference is that the number of points $B(x,y):=|L(\alpha, N) \cap ([0,x) \times [0,y))|$ which fall into the box $[0,x) \times [0,y)$ satisfies
\[ B \left( x, \frac{n+1}{N} \right) - (n+1) x = \sum_{\ell =0}^n \left( I_{[0,x)}(\{ \ell \alpha \}) -x \right) , \]
which is not a mean zero function. Its integral ($0$th Fourier coefficient) is
\[ \int_0^1 \left( B \left( x, \frac{n+1}{N} \right) - (n+1) x \right) \, \mathrm{d} x = \sum_{\ell =0}^n \left( \frac{1}{2} - \{ \ell \alpha \} \right) =T_n, \]
which introduces the extra terms $N^{-1} \sum_{n=0}^{N-1} T_n/2$ resp.\ $N^{-1} \sum_{n=0}^{N-1} T_n^2$ when the Parseval formula is applied to the analogue of $R$ resp.\ $M$ as above. For the convenience of the reader we mention that the analogue of formula \eqref{dsquareformula} is
\[ \begin{split} D_2^2 (L(\alpha, N)) = &\frac{1}{N} \sum_{n=0}^{N-1} \left( T_n^2+\frac{1}{2} T_n \right) + \frac{1}{N} \sum_{n=0}^{N-1} \sum_{m=1}^{q_{K-1}-1} \frac{\sin^2((n+1)m \pi \alpha)}{2\pi^4 m^2 \| m \alpha \|^2} \\ &+ \xi_L(\alpha, N) \pm \left( \sum_{k=0}^{K-1} \frac{a_{k+1}}{8q_k} + 2.78 \right) , \end{split} \]
where
\[ \xi_L (\alpha, N) := \frac{1}{N} \sum_{n=0}^{N-1} \sum_{m=q_{K-1}}^{q_K-1} \frac{\sin^2((n+1)m \pi \alpha)}{2\pi^4 m^2 \| m \alpha \|^2} . \]
\end{proof}

\subsection{Proof of Proposition \ref{simpleparsevalprop}}\label{section2.4}

The following lemma is a simpler form of formula \eqref{Tnvariance}, but it applies without any assumption on the partial quotients. As modifying the proof of \eqref{Tnvariance} is not entirely straightforward, we include the details.
\begin{lem}\label{Tnlemma} For any $K \ge 1$,
\[ \frac{1}{q_K} \sum_{n=0}^{q_K-1} (T_n-E_{q_K})^2 \ll \sum_{k=1}^K a_k^2 \]
with a universal implied constant.
\end{lem}

\begin{proof} For the sake of readability, set $p=p_K$ and $q=q_K$. For any integer $1 \le \ell \le q-1$, we have $\| \ell p/q \| \ge 1/q$ and $|\ell \alpha - \ell p/q| \le q |\alpha -p/q| < 1/q$. Thus there is no integer between $\ell p/q$ and $\ell \alpha$, hence
\[ \left| \{ \ell \alpha \} - \left\{ \frac{\ell p}{q} \right\} \right| \le \left| \ell \alpha - \frac{\ell p}{q} \right| < \frac{1}{q} . \]
Consequently, for all $0 \le n \le q-1$,
\[ T_n=\sum_{\ell=0}^n \left( \frac{1}{2} - \{ \ell \alpha \} \right) = \sum_{\ell =0}^n \left( \frac{1}{2} - \frac{1}{2q} - \left\{ \frac{\ell p}{q} \right\} \right) + O(1) . \]
Introducing
\[ T_n^* := \sum_{\ell =0}^n \left( \frac{1}{2} - \frac{1}{2q} - \left\{ \frac{\ell p}{q} \right\} \right) \quad \textrm{and} \quad E_q^* := \frac{1}{q} \sum_{n=0}^{q-1} T_n^*, \]
we thus have $T_n -E_q = T_n^* - E_q^*+O(1)$. Therefore $q^{-1} \sum_{n=0}^{q-1} (T_n-E_q)^2 \ll q^{-1} \sum_{n=0}^{q-1} (T_n^*-E_q^*)^2 +1$, and it remains to estimate the latter.

The rest of the proof is based on Fourier analysis on the finite cyclic group $\mathbb{Z}_q$, which we identify by $\{ 0,1,\dots, q-1 \}$. Elementary calculations show that
\[ \sum_{x=0}^{q-1} \left( \frac{1}{2} - \frac{1}{2q} - \left\{ \frac{x}{q} \right\} \right) e^{-2 \pi i m x/q} = \left\{ \begin{array}{ll} 0 & \textrm{if } m=0, \\ 1/(1-e^{-2 \pi i m/q}) & \textrm{if } 1 \le m \le q-1 . \end{array} \right. \]
Therefore by Fourier inversion on $\mathbb{Z}_q$,
\[ \frac{1}{2} - \frac{1}{2q} - \left\{ \frac{x}{q} \right\} = \frac{1}{q} \sum_{m=1}^{q-1} \frac{e^{2 \pi i m x/q}}{1-e^{-2 \pi i m /q}} , \qquad x \in \mathbb{Z} . \]
We can thus write $T_n^*$ as
\[ T_n^* = \frac{1}{q} \sum_{m=1}^{q-1} \sum_{\ell =0}^n \frac{e^{2 \pi i m \ell p /q}}{1-e^{-2 \pi i m/q}} = \frac{1}{q} \sum_{m=1}^{q-1} \frac{1-e^{2 \pi i m (n+1)p/q}}{(1-e^{- 2 \pi i m/q}) (1-e^{2 \pi i m p/q})} . \]
Letting $B=q^{-1} \sum_{m=1}^{q-1} 1/(1-e^{-2 \pi i m/q})(1-e^{2 \pi i mp/q})$, we have
\[ \frac{1}{q} \sum_{n=0}^{q-1} (T_n^*-E_q^*)^2 \le \frac{1}{q} \sum_{n=0}^{q-1} |T_n^*-B|^2 = \frac{1}{q} \sum_{n=0}^{q-1} \frac{1}{q^2} \left| \sum_{m=1}^{q-1} \frac{e^{2 \pi i m (n+1)p/q}}{(1-e^{-2 \pi i m/q})(1-e^{2 \pi i m p/q})} \right|^2 . \]
Expanding the square shows that here
\[ \begin{split} \bigg| \sum_{m=1}^{q-1} &\frac{e^{2 \pi i m (n+1)p/q}}{(1-e^{-2 \pi i m/q})(1-e^{2 \pi i m p/q})} \bigg|^2 \\ & \hspace{20mm} = \sum_{m=1}^{q-1} \frac{1}{|1-e^{-2 \pi i m/q}|^2 |1-e^{2 \pi i m p/q}|^2} \\ & \hspace{25mm} + \sum_{\substack{m_1, m_2=1 \\ m_1 \neq m_2}}^{q-1} \frac{e^{2 \pi i (m_1-m_2)(n+1)p/q}}{(1-e^{- 2 \pi i m_1 /q})(1-e^{2 \pi i m_1 p/q})(1-e^{2 \pi i m_2 /q})(1-e^{-2 \pi i m_2 p/q})}. \end{split} \]
As $\sum_{n=0}^{q-1} e^{2 \pi i (m_1-m_2)(n+1)p/q} =0$ for all $m_1 \neq m_2$, the contribution of the off-diagonal terms is zero. Formula \eqref{diophantineevaluation} thus leads to
\[ \frac{1}{q} \sum_{n=0}^{q-1} (T_n^*-E_q^*)^2 \le \frac{1}{q^2} \sum_{m=1}^{q-1} \frac{1}{|1-e^{-2 \pi i m/q}|^2 |1-e^{2 \pi i mp/q}|^2} \ll \sum_{m=1}^{q-1} \frac{1}{m^2 \| mp/q \|^2} \ll \sum_{k=1}^K a_k^2, \]
as claimed.
\end{proof}

\begin{proof}[Proof of Proposition \ref{simpleparsevalprop}] By Proposition \ref{parsevalprop}, for any $q_{K-1} \le N \le q_K$ we have
\[ D_2^2(S(\alpha, N)) \ll \sum_{m=1}^{q_K-1} \frac{1}{m^2 \| m \alpha \|^2} + \sum_{k=0}^{K-1} \frac{a_{k+1}}{q_k} + \sum_{k=0}^{K-2} \frac{a_{k+1}^3 q_k}{N} . \]
Here $a_{k+1}^3 q_k /N \le a_{k+1}^2$, hence formula \eqref{diophantineevaluation} yields $D_2^2(S(\alpha, N)) \ll \sum_{k=1}^K a_k^2$, as claimed. Using Lemma \ref{Tnlemma} and formula \eqref{EqK} we also deduce that for $N=q_K$,
\[ \frac{1}{q_K} \sum_{n=0}^{q_K-1} \left( T_n^2+\frac{1}{2} T_n \right) = \frac{1}{q_K} \sum_{n=0}^{q_K-1} (T_n-E_{q_K})^2 + E_{q_K}^2 + \frac{1}{2} E_{q_K} \ll \sum_{k=1}^K a_k^2 + \left( \sum_{k=1}^K (-1)^k a_k \right)^2 , \]
and the upper bound for $D_2^2(L(\alpha, q_K))$ follows.

Next, we prove the lower bounds. Let $c>0$ resp.\ $C>0$ denote suitably small resp.\ large universal constants whose values change from line to line. By Proposition \ref{parsevalprop} and formula \eqref{diophantineevaluation}, for $N=q_K$ we have
\[ \begin{split} D_2^2 (S(\alpha, q_K)) &\ge \sum_{m=1}^{q_K-1} \frac{1}{4 \pi^4 m^2 \| m \alpha \|^2} - \frac{\zeta (3)}{16 \pi^4 q_K} \sum_{k=0}^{K-1} (a_{k+1}+2)^3 q_k - C \sum_{k=1}^K a_k \\ &\ge \left( \frac{1}{360} - \frac{\zeta (3)}{16 \pi^4} \right) \sum_{k=1}^K a_k^2 -C \sum_{k=1}^K a_k . \end{split} \]
The point is that $1/360>\zeta (3)/(16 \pi^4)$, i.e.\ the coefficient of $a_k^2$ is positive. The contribution of all $k$ such that $a_k \ll 1$ is $\ll K$, and for all other terms $a_k^2$ dominates $a_k$. Therefore $D_2^2 (S(\alpha, q_K)) \ge c \sum_{k=1}^K a_k^2 -CK$. On the other hand, by Roth's theorem we also have $D_2^2 (S(\alpha, q_K)) \gg \log q_K \gg K$. Taking a suitable weighted average of the previous two inequalities establishes the lower bound $D_2^2(S(\alpha, q_K)) \ge c \sum_{k=1}^K a_k^2$.

From Proposition \ref{parsevalprop} we similarly deduce
\[ D_2^2(L(\alpha, q_K)) \ge \frac{1}{q_K} \sum_{n=1}^{q_K-1} \left( T_n-E_{q_K} \right)^2 + E_{q_K}^2 + c \sum_{k=1}^K a_k^2 . \]
Here $q_K^{-1} \sum_{n=0}^{q_K-1} (T_n-E_{q_K})^2 \ge 0$, and the lower bound for $D_2^2(L(\alpha, q_K))$ follows from formula \eqref{EqK}.
\end{proof}

\subsection{Proof of Remark \ref{LalphaNremark}}

Let $\alpha$ be an irrational such that $a_k \ll \sqrt{k}/\log^2 k$. For any $q_{K-1} \le N \le q_K$ we then have $\max_{|k-K| \ll \log K} a_k^2 \cdot (\log \log N)^4 \ll K$, hence formulas \eqref{Tnvariance} and \eqref{diophantineevaluation} give
\[ \frac{1}{N} \sum_{n=0}^{N-1} (T_n-E_N)^2 = \sum_{m=1}^{q_K-1} \frac{1}{8 \pi^4 m^2 \| m \alpha \|^2} +O(K) \ll \sum_{k=1}^K a_k^2 . \]
Using this fact instead of Lemma \ref{Tnlemma} in the proof of Proposition \ref{simpleparsevalprop}, we deduce that $D_2^2(L(\alpha, N)) \ll \sum_{k=1}^K a_k^2 + (\sum_{k=1}^K (-1)^k a_k)^2$ holds for all $q_{K-1} \le N \le q_K$ (instead of only for $N=q_K$). In particular, the equivalence stated in Remark \ref{LalphaNremark} follows.

\section{Typical irrationals}\label{typicalirrationalsection}

\subsection{Asymptotics almost everywhere}

Let us recall certain basic facts about the statistics of the partial quotients of a typical irrational number. Let $\varphi$ be a positve nondecreasing function on $(0,\infty)$, and let $A_K=\max_{1 \le k \le K} a_k$. It is well known that for a.e.\ $\alpha$ we have $\log q_k \sim \frac{\pi^2}{12 \log 2} k$, and that $a_k \le \varphi (k)$ for all but finitely many $k$ if and only if $\sum_{n=1}^{\infty} 1/\varphi (n)< \infty$. A classical result of Diamond and Vaaler \cite{DV} on trimmed sums states that for a.e.\ $\alpha$,
\begin{equation}\label{diamondvaaler}
\frac{\sum_{k=1}^K a_k - A_K}{K \log K} \to \frac{1}{\log 2} \qquad \textrm{as } K \to \infty .
\end{equation}

\begin{proof}[Proof of Theorem \ref{aeasymptotictheorem}] For any $N \ge 2$, let $K_N(\alpha)$ be the positive integer for which $q_{K_N(\alpha) -1} < N \le q_{K_N(\alpha )}$. In particular, for a.e.\ $\alpha$ we have $K_N(\alpha ) \sim \frac{12 \log 2}{\pi^2} \log N$, where $\frac{12 \log 2}{\pi^2}=0.8427\dots$.

\noindent\textbf{(i)} Assume that $\sum_{n=1}^{\infty} 1/\varphi(n)<\infty$. As observed in the Introduction, by a classical discrepancy estimate for the sequence $\{ n \alpha \}$ \cite[p.\ 52]{DT}, we have
\[ \begin{split} &D_2 (S(\alpha, N)) \ll D_{\infty} (L(\alpha, N)) \ll \sum_{k=1}^{K_N(\alpha)} a_k, \\ &D_2 (L(\alpha, N)) \ll D_{\infty} (L(\alpha, N)) \ll \sum_{k=1}^{K_N(\alpha)} a_k . \end{split} \]
The asymptotic relation \eqref{diamondvaaler} of Diamond and Vaaler shows that for a.e.\ $\alpha$,
\[ \begin{split} &D_2 (S(\alpha, N)) \le C \sum_{k=1}^{K_N(\alpha)} a_k = C A_{K_N(\alpha)} + O(K_N(\alpha) \log K_N(\alpha)) , \\ &D_2 (L(\alpha, N)) \le C \sum_{k=1}^{K_N(\alpha)} a_k = C A_{K_N(\alpha)} +  O(K_N(\alpha) \log K_N(\alpha) ) \end{split} \]
with a universal constant $C>0$. Here $A_{K_N(\alpha)} \le \varphi(K_N(\alpha))$ and $K_N(\alpha) \le \log N$ for all but finitely many $N$. Therefore $D_2(S(\alpha, N)) \le C \varphi (\log N)+O(\log N \log \log N)$ and $D_2(L(\alpha, N)) \le C \varphi (\log N)+O(\log N \log \log N)$ with implied constants depending only on $\alpha$ and $\varphi$. The factor $C$ can be removed by repeating the argument with $\varphi(x)/C$ instead of $\varphi(x)$.

\noindent\textbf{(ii)} Assume that $\sum_{n=1}^{\infty} 1/\varphi(n)=\infty$. By Proposition \ref{simpleparsevalprop}, we have
\[ D_2(S(\alpha, q_K)) \ge c \bigg( \sum_{k=1}^K a_k^2 \bigg)^{1/2} \ge c A_K \quad \textrm{and} \quad D_2(L(\alpha, q_K)) \ge c \bigg( \sum_{k=1}^K a_k^2 \bigg)^{1/2} \ge c A_K \]
with a universal constant $c>0$. Here $A_K \ge \varphi(K)$ for infinitely many $K$, and $K \ge (\log q_K)/2$ for all but finitely many $K$. Hence $D_2(S(\alpha, q_K)) \ge c \varphi ((\log q_K)/2)$ and $D_2(L(\alpha, q_K)) \ge c \varphi ((\log q_K)/2)$ for infinitely many $K$. Repeating the argument with $\varphi (2x)/c$ instead of $\varphi (x)$, we deduce that $D_2(S(\alpha, q_K)) \ge \varphi (\log q_K)$ and $D_2(L(\alpha, q_K)) \ge \varphi (\log q_K)$ for infinitely many $K$, as claimed.
\end{proof}

\subsection{Limit distribution}\label{limitdistributionsection}

Let $\lambda$ be the Lebesgue measure, and $\nu (B)= (1/\log 2) \int_B 1/(1+x) \, \mathrm{d} x$ ($B \subseteq [0,1]$ Borel) the Gauss measure. If $\alpha$ is chosen randomly from $[0,1]$ with distribution $\nu$, then its partial quotients are identically distributed random variables with distribution
\[ \nu \left( \left\{ \alpha \in [0,1] \, : \, a_k = n \right\} \right) = \frac{1}{\log 2} \log \left( 1+\frac{1}{n(n+2)} \right) , \qquad k,n \ge 1. \]
If $\alpha$ is chosen randomly from $[0,1]$ with distribution either $\lambda$ or $\nu$, then the sequence $a_k$ is $\psi$-mixing with exponential rate \cite[p.\ 119]{IK}.

To find the limit distribution of $D_2^2(S(\alpha, N))/\log^2 N$, we shall need more sophisticated facts about the partial quotients of a typical irrational, which we now gather. Most importantly, a special case of a limit distribution theorem of Samur \cite{SA} (see also \cite{BO1}) states that if $\mu$ is a Borel probability measure on $[0,1]$ which is absolutely continuous with respect to the Lebesgue measure, then for any $t \ge 0$,
\begin{equation}\label{Samur}
\mu \left( \left\{ \alpha \in [0,1] \, : \, \frac{2 \log^2 2}{\pi K^2} \sum_{k=1}^K a_k^2 \le t \right\} \right) \to \int_0^t \frac{e^{-1/(2x)}}{\sqrt{2 \pi} x^{3/2}} \, \mathrm{d}x \qquad \textrm{as } K \to \infty .
\end{equation}
If $\mu$ is either $\lambda$ or $\nu$, then general results of Heinrich \cite{HE} on $\psi$-mixing random variables imply the rate of convergence
\begin{equation}\label{Heinrich1}
\sup_{t \ge 0} \left| \mu \left( \left\{ \alpha \in [0,1] \, : \, \frac{2 \log^2 2}{\pi K^2} \sum_{k=1}^K a_k^2 \le t \right\} \right) - \int_0^t \frac{e^{-1/(2x)}}{\sqrt{2 \pi} x^{3/2}} \, \mathrm{d}x \right| \ll \frac{1}{K^{1-\varepsilon}}
\end{equation}
with an arbitrary $\varepsilon >0$ and an implied constant depending only on $\varepsilon$. The corresponding result for $\sum_{k=1}^K a_k$ in the Gauss measure is also due to Heinrich:
\[ \sup_{t \in \mathbb{R}} \left| \nu \left( \left\{ \alpha \in [0,1] \, : \, \frac{1}{K} \sum_{k=1}^K a_k - \frac{\log K -\gamma}{\log 2} \le t \right\} \right) - F(t) \right| \ll \frac{\log^2 K}{K} , \]
where $\gamma$ is the Euler--Mascheroni constant, and $F(t)$ is the distribution function of the law with characteristic function
\[ \int_{\mathbb{R}} e^{itx} \, \mathrm{d}F(t) = \exp \left( - \frac{\pi}{2 \log 2} |x| \left( 1+\frac{2i}{\pi} \mathrm{sgn}(x) \log |x| \right) \right) . \]
Note that this is a stable law with stability parameter $1$ (and skewness parameter $1$). Hence $1-F(t) \ll 1/t$ as $t \to \infty$, and we immediately obtain
\begin{equation}\label{sumakbound}
\nu \left( \left\{ \alpha \in [0,1] \, : \, \frac{1}{K} \sum_{k=1}^K a_k \ge t + \frac{\log K}{\log 2} \right\} \right) \ll \frac{1}{t} + \frac{\log^2 K}{K} \qquad \textrm{as } t \to \infty .
\end{equation}
The final ingredient is a similar estimate for the convergent denominators: with a large enough universal constant $C>0$,
\begin{equation}\label{qKbound}
\nu \left( \left\{ \alpha \in [0,1] \, : \, \left| \log q_K - \frac{\pi^2}{12 \log 2} K \right| \ge C \sqrt{K \log K} \right\} \right) \ll \frac{1}{\sqrt{K}} .
\end{equation}
This follows from the fact that $\log q_K$ satisfies the central limit theorem with rate $O(1/\sqrt{K})$, as shown by Morita \cite{MO}. We mention that a better upper bound can be deduced from the large deviation inequality of Takahasi \cite{TA}, but \eqref{qKbound} suffices for our purposes.

\begin{proof}[Proof of Theorem \ref{irrationallimitdistributiontheorem}] Throughout the proof, $C>0$ is a large universal constant whose value changes from line to line, and $Y_i=Y_i(\alpha, N)$, $i=1,2,\ldots$ are error terms. For any $N \ge 2$, let $K_N(\alpha)$ be the positive integer for which $q_{K_N(\alpha) -1} < N \le q_{K_N(\alpha )}$.

Proposition \ref{parsevalprop} and formula \eqref{diophantineevaluation} show that we can write
\[ D_2^2 (S(\alpha, N)) = \frac{1}{360} \sum_{k=1}^{K_N(\alpha)-1} a_k^2 + Y_1, \quad \textrm{where} \quad |Y_1| \le \frac{1}{180} a_{K_N(\alpha)}^2+ C \sum_{k=1}^{K_N(\alpha)} a_k + \frac{C}{N} \sum_{k=0}^{K_N(\alpha) -2} a_{k+1}^3 q_k . \]
Using the general fact $q_{k+2}/q_k \ge 2$, we estimate the last error term as
\[ \begin{split} \frac{1}{N} \sum_{k=0}^{K_N(\alpha)-2} a_{k+1}^3 q_k &\le \frac{1}{N} \sum_{k=1}^{K_N(\alpha)-1} a_k^2 q_k \\ &\le \sum_{k=1}^{K_N(\alpha) - 100 \log K_N(\alpha)} a_k^2 \frac{q_k}{q_{K_N(\alpha)-1}} + \sum_{k=K_N(\alpha) - 100 \log K_N(\alpha)}^{K_N(\alpha) -1} a_k^2 \frac{q_k}{q_{K_N(\alpha) -1}} \\ &\le \frac{1}{K_N(\alpha)^{10}} \sum_{k=1}^{K_N(\alpha )} a_k^2 + \sum_{k=K_N(\alpha) - 100 \log K_N(\alpha)}^{K_N(\alpha) -1} a_k^2 . \end{split} \]
This leads to the simplified form $D_2^2(S(\alpha, N)) = (1/360) \sum_{k=1}^{K_N(\alpha)} a_k^2 + Y_2$, where
\[ |Y_2| \le \frac{C}{K_N(\alpha)^{10}} \sum_{k=1}^{K_N(\alpha)} a_k^2 + C\sum_{k=K_N(\alpha) -100 \log K_N(\alpha)}^{K_N(\alpha)} a_k^2 + C\sum_{k=1}^{K_N(\alpha)} a_k .  \]
Set $\overline{K}=\lceil \frac{12 \log 2}{\pi^2} \log N \rceil$. The estimate \eqref{qKbound} states that
\[ \nu \left( \left\{ \alpha \in [0,1] \, : \, \left| \log q_{\overline{K}} - \frac{\pi^2}{12 \log 2} \overline{K} \right| \ge C \sqrt{\overline{K} \log \overline{K}} \right\} \right) \ll \frac{1}{\sqrt{\overline{K}}} . \]
By the definition of $K_N(\alpha)$ and $\overline{K}$, this immediately gives
\[ \nu \left( \left\{ \alpha \in [0,1] \, : \, |K_N(\alpha ) - \overline{K}| \ge C \sqrt{\overline{K} \log \overline{K}} \right\} \right) \ll \frac{1}{\sqrt{\overline{K}}} . \]
Roughly speaking, this means that we can replace $K_N(\alpha)$ by $\overline{K}$ in the above formulas; the point is that the latter does not depend on $\alpha$. More precisely, outside a set of $\nu$-measure $\ll 1/\sqrt{\overline{K}}$, we have $D_2^2(S(\alpha, N)) = (1/360) \sum_{k=1}^{\overline{K}} a_k^2+Y_3$, where
\[ |Y_3| \le \frac{C}{\overline{K}^{10}} \sum_{k=1}^{2 \overline{K}} a_k^2 + C \sum_{k=\overline{K}-C\sqrt{\overline{K} \log \overline{K}}}^{\overline{K} + C \sqrt{\overline{K} \log \overline{K}}} a_k^2 +C \sum_{k=1}^{2\overline{K}} a_k . \]
Since $5 \pi^3 /\log^2 N = 720 \log^2 2 /(\pi \overline{K}^2) +O(1/\overline{K}^3)$, normalizing the previous formula leads to the fact that outside a set of $\nu$-measure $\ll 1/\sqrt{\overline{K}}$,
\[ 5 \pi^3 \frac{D_2^2 (S(\alpha, N))}{\log^2 N} = \frac{2 \log^2 2}{\pi \overline{K}^2} \sum_{k=1}^{\overline{K}} a_k^2 +Y_4 , \]
where
\[ |Y_4| \le \frac{C}{\overline{K}^3} \sum_{k=1}^{2 \overline{K}} a_k^2 + \frac{C}{\overline{K}^2} \sum_{k=\overline{K}-C\sqrt{\overline{K} \log \overline{K}}}^{\overline{K} + C \sqrt{\overline{K} \log \overline{K}}} a_k^2 +\frac{C}{\overline{K}^2} \sum_{k=1}^{2\overline{K}} a_k . \]
We now estimate the three error terms in the previous formula. The limit distribution with rate of Heinrich \eqref{Heinrich1} gives
\[ \nu \left( \left\{ \alpha \in [0,1] \, : \, \frac{1}{\overline{K}^3} \sum_{k=1}^{2 \overline{K}} a_k^2 \ge \frac{1}{\overline{K}^{1/3}} \right\} \right) \ll \int_{\mathrm{const} \cdot \overline{K}^{2/3}}^{\infty} \frac{e^{-1/(2x)}}{\sqrt{2 \pi} x^{3/2}} \, \mathrm{d} x + \frac{1}{\overline{K}^{1-\varepsilon}} \ll \frac{1}{\overline{K}^{1/3}} . \]
Since the sequence $a_k$ is strictly stationary, we similarly deduce
\[ \begin{split} \nu \Bigg( \Bigg\{ \alpha \in [0,1] \, : \, \frac{1}{\overline{K}^2} \sum_{k=\overline{K}-C \sqrt{\overline{K} \log \overline{K}}}^{\overline{K}+C \sqrt{\overline{K} \log \overline{K}}} &a_k^2 \ge \frac{(\log \overline{K})^{1/3}}{\overline{K}^{1/3}} \Bigg\} \Bigg) \\ &= \nu \left( \left\{ \alpha \in [0,1] \, : \, \frac{1}{\overline{K}^2} \sum_{k=1}^{C \sqrt{\overline{K} \log \overline{K}}} a_k^2 \ge \frac{(\log \overline{K})^{1/3}}{\overline{K}^{1/3}} \right\} \right) \\ &\ll \int_{\mathrm{const} \cdot \overline{K}^{2/3}/(\log \overline{K})^{2/3}}^{\infty} \frac{e^{-1/(2x)}}{\sqrt{2 \pi} x^{3/2}} \, \mathrm{d} x + \frac{1}{\overline{K}^{1/2-\varepsilon}} \\ &\ll \frac{(\log \overline{K})^{1/3}}{\overline{K}^{1/3}}. \end{split} \]
Finally, formula \eqref{sumakbound} gives
\[ \nu \left( \left\{ \alpha \in [0,1] \, : \, \frac{1}{\overline{K}^2} \sum_{k=1}^{2\overline{K}} a_k \ge \frac{1}{\overline{K}^{1/3}} \right\} \right) \ll \frac{1}{\overline{K}^{2/3}} . \]
By the previous three estimates, we can finally write
\begin{equation}\label{D2inmeasure}
5 \pi^3 \frac{D_2^2 (S(\alpha, N))}{\log^2 N} = \frac{2 \log^2 2}{\pi \overline{K}^2} \sum_{k=1}^{\overline{K}} a_k^2 +Y_5,
\end{equation}
where
\begin{equation}\label{Y5bound}
\nu \left( \left\{ \alpha \in [0,1] \, : \, |Y_5| \ge C \frac{(\log \overline{K})^{1/3}}{\overline{K}^{1/3}} \right\} \right) \le C \frac{(\log \overline{K})^{1/3}}{\overline{K}^{1/3}} .
\end{equation}

The proof of the theorem is now immediate. Assume first, that $\mu$ is absolutely continuous with respect to the Lebesgue measure. The theorem of Samur \eqref{Samur} ensures that the main term in \eqref{D2inmeasure} converges in distribution to the standard L\'evy distribution as $N$, and hence $\overline{K}$, goes to infinity. Since $Y_5 \to 0$ in $\nu$-measure, the same holds also in $\mu$-measure, and the convergence to the standard L\'evy distribution remains true for the left hand side of \eqref{D2inmeasure}. This finishes the proof for a general absolutely continuous measure $\mu$.

Next, let $\mu$ be either $\lambda$ or $\nu$. Then the sequence $a_k$ is $\psi$-mixing with exponential rate, and the limit distribution with rate of Heinrich \eqref{Heinrich1} ensures that the main term in \eqref{D2inmeasure} converges to the standard L\'evy distribution with rate $\ll 1/\overline{K}^{1-\varepsilon}$. The estimate \eqref{Y5bound}, which holds also with $\lambda$ in place of $\nu$, together with the trivial fact that the distribution function of the L\'evy distribution is Lipschitz, shows that this convergence remains true for the left hand side of \eqref{D2inmeasure} with the rate $\ll (\log \overline{K})^{1/3} / \overline{K}^{1/3}$. This finishes the proof of the rate of convergence for $\lambda$ and $\nu$.
\end{proof}

\section{Typical rationals}\label{typicalrationalsection}

Let $F_Q$ denote the set of all reduced fractions in $[0,1]$ with denominator at most $Q$, and let us write every $p/q \in F_Q$ in the form $p/q=[0;a_1, \ldots, a_r]$. It does not matter which of the two possible expansions is chosen. Note that the partial quotients $a_1=a_1(p/q), \ldots, a_r=a_r(p/q)$ as well as the length $r=r(p/q)$ is a function of $p/q$. For the sake of simplicity, we use the convention $a_k=0$ if $k>r$.

The proof of Theorem \ref{rationallimitdistributiontheorem} is based on recent results of Bettin and Drappeau on the limit distribution of power sums of the partial quotients; they are perfect analogues of the results for typical irrationals mentioned in Section \ref{limitdistributionsection}.
\begin{lem}[Bettin--Drappeau \cite{BD}]\label{BDlemma} For any $Q \ge 2$ and $\varepsilon>0$,
\begin{equation}\label{Bettin1}
\sup_{t \ge 0} \left| \frac{1}{|F_Q|} \left| \left\{ \frac{p}{q} \in F_Q \, : \, \frac{\pi^3}{72 \log^2 Q} \sum_{k=1}^r a_k^2 \le t \right\} \right| - \int_0^t \frac{e^{-1/(2x)}}{\sqrt{2 \pi} x^{3/2}} \, \mathrm{d} x \right| \ll \frac{1}{(\log Q)^{1-\varepsilon}}
\end{equation}
and
\[ \sup_{t \in \mathbb{R}} \left| \frac{1}{|F_Q|} \left| \left\{ \frac{p}{q} \in F_Q \, : \, \frac{1}{\log Q} \sum_{k=1}^r a_k - \frac{\log \log Q - \gamma}{\pi^2/12} \le t \right\} \right| - G(t) \right| \ll \frac{1}{(\log Q)^{1-\varepsilon}} \]
with implied constants depending only on $\varepsilon$. Here $\gamma$ is the Euler--Mascheroni constant, and $G(t)$ is the distribution function of the law with characteristic function
\[ \int_{\mathbb{R}} e^{itx} \, \mathrm{d} G(t) = \exp \left( - \frac{6}{\pi} |x| \left( 1 + \frac{2i}{\pi} \mathrm{sgn}(x) \log |x| \right) \right) . \]
\end{lem}
\noindent The second limit distribution in Lemma \ref{BDlemma} immediately yields
\begin{equation}\label{Bettin2}
\frac{1}{|F_Q|} \left| \left\{ \frac{p}{q} \in F_Q \, : \, \frac{1}{\log Q} \sum_{k=1}^r a_k \ge t + \frac{\log \log Q}{\pi^2/12} \right\} \right| \ll \frac{1}{t} + \frac{1}{(\log Q)^{1-\varepsilon}} \quad \textrm{as } t \to \infty .
\end{equation}
Note that \eqref{Bettin1} was stated in \cite{BD} with the rate $\ll 1/(\log \log Q)^{1-\varepsilon}$, but the methods of that paper actually give $\ll 1/(\log Q)^{1-\varepsilon}$. For the sake of completeness, we deduce \eqref{Bettin1} as stated here in Section \ref{section4.1}. We now prove a lemma which will serve as a substitute for the fact that the partial quotients are not exactly identically distributed, and then prove Theorem \ref{rationallimitdistributiontheorem}.

\begin{lem}\label{partialquotientlemma} For any positive integers $Q,k,t$, we have
\[ \left| \left\{ \frac{p}{q} \in F_Q \, : \, a_k \ge t \right\} \right| \le \frac{2Q^2}{t} . \]
\end{lem}

\begin{proof} Assume first, that $k=1$. Note that $a_1 \ge t$ implies that $0 < p/q \le 1/t$. In particular, for each $1 \le q \le Q$ there are at most $q/t$ possible numerators $p$, hence
\begin{equation}\label{k=1case}
\left| \left\{ \frac{p}{q} \in F_Q \, : \, a_1 \ge t \right\} \right| \le \sum_{q=1}^Q \frac{q}{t} \le \frac{Q^2}{t} .
\end{equation}
Next, assume that $k \ge 2$. Let $\mathrm{denom}(x)$ denote the denominator of a rational $x$ (in its reduced form). From the recursion satisfied by the denominator of the convergents one readily deduces the supermultiplicative property
\[ \mathrm{denom}([0;a_1, \ldots, a_r]) \ge \mathrm{denom}([0;a_1, \ldots,a_{k-1}]) \cdot \mathrm{denom}([0;a_k, \ldots , a_r]) . \]
For any fixed positive integers $b_1, \ldots, b_{k-1}$ we thus obtain
\[ \left| \left\{ \frac{p}{q} \in F_Q \, : \, a_1=b_1, \ldots, a_{k-1}=b_{k-1}, \,\, a_k \ge t \right\} \right| \le \left| \left\{ \frac{p}{q} \in F_{Q/\mathrm{denom}([0;b_1, \ldots, b_{k-1}])} \, : \, a_1 \ge t \right\} \right| . \]
Summing over $b_1, \ldots, b_{k-1}$ and applying \eqref{k=1case} leads to
\[ \left| \left\{ \frac{p}{q} \in F_Q \, : \, a_k \ge t \right\} \right| \le \sum_{b_1, \ldots, b_{k-1}=1}^{\infty} \frac{Q^2}{t (\mathrm{denom}([0;b_1, \dots, b_{k-1}]))^2} . \]
Recall that the set of real numbers $[0;c_1, c_2, \ldots ] \in [0,1]$ such that $c_1=b_1, \ldots, c_{k-1}=b_{k-1}$ is an interval whose length is at least $1/(2 \, \mathrm{denom}([0;b_1, \ldots, b_{k-1}])^2)$. Since these are pairwise disjoint intervals, we have
\[ \sum_{b_1, \ldots, b_{k-1}=1}^{\infty} \frac{1}{(\mathrm{denom}([0;b_1, \ldots, b_{k-1}]))^2} \le 2, \]
and the claim follows.
\end{proof}

\begin{proof}[Proof of Theorem \ref{rationallimitdistributiontheorem}] Throughout the proof, $C>0$ is a large universal constant whose value changes from line to line, and $Z_i=Z_i (p/q)$, $i=1,2$ are error terms.

Proposition \ref{parsevalprop} and formula \eqref{diophantineevaluation} show that we can write
\[ D_2^2 (S(p/q,q)) = \frac{1}{360} \sum_{k=1}^r a_k^2 +Z_1, \quad \textrm{where} \quad |Z_1| \le C \sum_{k=1}^r a_k + \frac{C}{q} \sum_{k=0}^{r-1} a_{k+1}^3 q_k . \]
Here $a_{k+1}^3 q_k \le a_{k+1}^2 q_{k+1}$, and $q_k/q=q_k/q_r \le 1/F_{r-k+1}$, where $F_{r-k+1}$ are the Fibonacci numbers. Hence normalizing the previous formula leads to
\[ 5 \pi^3 \frac{D_2^2 (S(p/q,q))}{\log^2 Q} = \frac{\pi^3}{72 \log^2 Q} \sum_{k=1}^r a_k^2 + Z_2, \quad \textrm{where} \quad |Z_2| \le \frac{C}{\log^2 Q} \sum_{k=1}^r a_k + \frac{C}{\log^2 Q} \sum_{k=1}^r \frac{a_k^2}{F_{r-k+1}} .   \]
The first error term can be estimated in measure using formula \eqref{Bettin2} as
\[ \frac{1}{|F_Q|} \left| \left\{ \frac{p}{q} \in F_Q \, : \, \frac{1}{\log^2 Q} \sum_{k=1}^r a_k \ge \frac{1}{(\log Q)^{1/2}} \right\} \right| \ll \frac{1}{(\log Q)^{1/2}} . \]
Note that the map reversing the order of the partial quotients $F_Q \to F_Q$, $[0;a_1, a_2, \ldots, a_r] \mapsto [0; a_r, \ldots, a_2, a_1]$ is a bijection; in fact, $[0;a_r,\ldots, a_2,a_1]$ is the reduced fraction $q_{r-1}/q_r$. Therefore the distribution of $(a_r, \ldots, a_2, a_1)$ is identical to that of $(a_1, a_2, \ldots, a_r)$, and we can apply Lemma \ref{partialquotientlemma} to estimate the second error term in measure as
\[ \begin{split} \frac{1}{|F_Q|} \left| \left\{ \frac{p}{q} \in F_Q \, : \, \frac{1}{\log^2 Q} \sum_{k=1}^r \frac{a_k^2}{F_{r-k+1}} \ge \frac{1}{(\log Q)^{1/2}} \right\} \right| &= \frac{1}{|F_Q|} \left| \left\{ \frac{p}{q} \in F_Q \, : \, \sum_{k=1}^r \frac{a_k^2}{F_k} \ge (\log Q)^{3/2} \right\} \right| \\ &\le \frac{1}{|F_Q|} \sum_{k=1}^{\infty} \left| \left\{ \frac{p}{q} \in F_Q \, : \, \frac{a_k^2}{F_k} \ge (\log Q)^{3/2} \right\} \right| \\ &\le \frac{1}{|F_Q|} \sum_{k=1}^{\infty} \frac{2Q^2}{F_k^{1/2} (\log Q)^{3/4}} \\ &\ll \frac{1}{(\log Q)^{3/4}} . \end{split} \]
Note that we used the convention $a_k=0$ if $k>r$, and the fact that $|F(Q)| \gg Q^2$. In particular,
\[ \frac{1}{|F_Q|} \left| \left\{ \frac{p}{q} \in F_Q \, : \, |Z_2| \ge \frac{1}{(\log Q)^{1/2}} \right\} \right| \ll \frac{1}{(\log Q)^{1/2}} , \]
and the limit distribution theorem \eqref{Bettin1} of Bettin and Drappeau yields
\[ \sup_{t \ge 0} \left| \frac{1}{|F_Q|} \left| \left\{ \frac{p}{q} \in F_Q \, : \, 5 \pi^3 \frac{D_2^2 (S(p/q,q))}{\log^2 Q} \le t \right\} \right| - \int_0^t \frac{e^{-1/(2x)}}{\sqrt{2 \pi} x^{3/2}} \, \mathrm{d} x \right| \ll \frac{1}{(\log Q)^{1/2}} . \]
The error of replacing $\log^2 Q$ by $\log^2 q$ is easily seen to be negligible compared to $1/(\log Q)^{1/2}$.
\end{proof}

\subsection{Proof of Lemma \ref{BDlemma}}\label{section4.1}

We now deduce the rate $\ll 1/(\log Q)^{1-\varepsilon}$ in \eqref{Bettin1}. Fix $\varepsilon >0$. Applying the main result \cite[Theorem 1.1]{BD} of Bettin and Drappeau to, in their notation, $\phi(x)=\lfloor 1/x \rfloor^2$ with $\alpha_0=1/2-\varepsilon$, we conclude that there exist constants $t_0,\delta>0$ such that for all $|t| < t_0$,
\begin{equation}\label{charfunction}
\frac{1}{|F_Q|} \sum_{p/q \in F_Q} \exp \left( it \sum_{k=1}^r a_k^2 \right) = \exp \left( U(t) \log Q + O \left( |t|^{1/2-\varepsilon} + Q^{-\delta} \right) \right) ,
\end{equation}
where
\[ U(t) = \frac{12}{\pi^2} \int_{0}^{1} \frac{e^{it\lfloor 1/x \rfloor^2} -1}{1+x} \, \mathrm{d}x + O \left( |t|^{1-\varepsilon} \right) = \frac{12}{\pi^2} \int_{1}^{\infty} \frac{e^{it \lfloor x \rfloor^2}-1}{x^2+x} \, \mathrm{d}x + O \left( |t|^{1-\varepsilon} \right) . \]
Here $t_0, \delta$ and the implied constants depend only on $\varepsilon$.

Our improvement in \eqref{Bettin1} comes from a more careful estimate for $U(t)$. Assume that $0<t<t_0$. Since $|\lfloor x \rfloor^2 -x^2| \le 2x$, the error of removing the integer part function is negligible:
\[ \left| \int_{1}^{\infty} \frac{e^{i t \lfloor x \rfloor^2} - e^{i t x^2}}{x^2+x} \, \mathrm{d}x \right| \le \int_{1}^{\infty} \frac{\min \{ 2 t x , 2 \}}{x^2+x} \, \mathrm{d} x \ll t \log \frac{1}{t} . \]
Therefore
\[ U(t) = \frac{12}{\pi^2} \int_{1}^{\infty} \frac{e^{itx^2}-1}{x^2+x} \, \mathrm{d} x +O(t^{1-\varepsilon}) = \frac{12 \sqrt{t}}{\pi^2} \int_{\sqrt{t}}^{\infty} \frac{e^{ix^2}-1}{x^2+\sqrt{t}x} \, \mathrm{d} x +O(t^{1-\varepsilon}) . \]
We now compare the remaining integral to its limit, the Fresnel-type integral $\int_{0}^{\infty} (e^{ix^2-1})/x^2 \, \mathrm{d} x = (i-1) \sqrt{2 \pi}/2$. We have
\[ \begin{split} \left| \int_{\sqrt{t}}^{\infty} \frac{e^{ix^2}-1}{x^2+\sqrt{t}x} \, \mathrm{d} x - \int_{0}^{\infty} \frac{e^{ix^2}-1}{x^2} \, \mathrm{d} x \right| &\le \left| \int_{0}^{\sqrt{t}} \frac{e^{ix^2}-1}{x^2} \, \mathrm{d} x \right| + \int_{\sqrt{t}}^{\infty} |e^{ix^2}-1| \cdot \left| \frac{1}{x^2+\sqrt{t}x} - \frac{1}{x^2} \right| \, \mathrm{d} x \\ &\le \int_{0}^{\sqrt{t}} 1 \, \mathrm{d} x + \int_{\sqrt{t}}^{\infty} \min \{ x^2, 2 \} \frac{\sqrt{t}}{x^3} \, \mathrm{d} x \\ &\ll \sqrt{t} \log \frac{1}{t} , \end{split} \]
hence $U(t)= \frac{6 \sqrt{2} \sqrt{t}}{\pi^{3/2}} (i-1) + O(t^{1-\varepsilon})$. The case of negative $t$ follows from complex conjugation, thus for $|t|<t_0$,
\begin{equation}\label{Utestimate}
U(t) = - \frac{6 \sqrt{2} |t|^{1/2}}{\pi^{3/2}} (1- i \mathrm{sgn}(t)) + O(|t|^{1-\varepsilon}) .
\end{equation}

Now let
\[ \varphi_1(t)= \frac{1}{|F_Q|} \sum_{p/q \in F_Q} \exp \left( i t \frac{\pi^3}{72 \log^2 Q} \sum_{k=1}^r a_k^2 \right) \]
and $\varphi_2(t)=\exp (-|t|^{1/2}(1-i \mathrm{sgn}(t)))$; the latter is the characteristic function of the standard L\'evy distribution. The Berry--Esseen inequality \cite[p.\ 142]{PE} states that the distance of these two distributions in the Kolmogorov metric is, with any $T>0$,
\[ \sup_{t \ge 0} \left| \frac{1}{|F_Q|} \left| \left\{ \frac{p}{q} \in F_Q \, : \, \frac{\pi^3}{72 \log^2 Q} \sum_{k=1}^r a_k^2 \le t \right\} \right| - \int_0^t \frac{e^{-1/(2x)}}{\sqrt{2 \pi} x^{3/2}} \, \mathrm{d} x \right| \ll \frac{1}{T} + \int_{0}^{T} \frac{|\varphi_1(t)-\varphi_2(t)|}{t} \, \mathrm{d} t . \]
Choose $T = \log Q$. Formulas \eqref{charfunction} and \eqref{Utestimate} show that for $|t| \le \log Q$,
\[ \begin{split} \varphi_1(t) &= \varphi_2(t) \exp \left( O \left( \left( \frac{|t|}{\log^2 Q} \right)^{1-\varepsilon} \log Q + \left( \frac{|t|}{\log^2 Q} \right)^{1/2-\varepsilon} + Q^{-\delta} \right) \right) \\ &= \varphi_2(t) \left( 1 + O \left( \frac{|t|^{1-\varepsilon}+|t|^{1/2-\varepsilon}}{(\log Q)^{1-2 \varepsilon}} + Q^{-\delta} \right) \right) . \end{split} \]
Using $|\varphi_2(t)|= e^{-|t|^{1/2}}$, this immediately yields
\[ |\varphi_1(t)-\varphi_2(t)| \ll e^{-|t|^{1/2}} \left( \frac{|t|^{1-\varepsilon}+|t|^{1/2-\varepsilon}}{(\log Q)^{1-2 \varepsilon}} + Q^{-\delta} \right) . \]
It is now easy to see that
\[ \int_{Q^{-100}}^{1} \frac{|\varphi_1(t)-\varphi_2(t)|}{t} \, \mathrm{d} t \ll \frac{1}{(\log Q)^{1-2 \varepsilon}} \quad \textrm{and} \quad \int_{1}^{\log Q} \frac{|\varphi_1(t)-\varphi_2(t)|}{t} \, \mathrm{d} t \ll \frac{1}{(\log Q)^{1-2 \varepsilon}} . \]
On the other hand, by a very rough estimate we have $\sum_{k=1}^r a_k^2 \le Q^3$, hence $|\varphi_1(t)-1| \ll |t| Q^3$. Clearly $|\varphi_2(t)-1| \ll |t|^{1/2}$, thus
\[ \int_{0}^{Q^{-100}} \frac{|\varphi_1(t)-\varphi_2(t)|}{t} \, \mathrm{d} t \ll \int_{0}^{Q^{-100}} \frac{t Q^3 + t^{1/2}}{t} \, \mathrm{d} t \ll Q^{-50} . \]
Therefore
\[ \frac{1}{\log Q} + \int_{0}^{\log Q} \frac{|\varphi_1(t)-\varphi_2(t)|}{t} \, \mathrm{d} t \ll \frac{1}{(\log Q)^{1-2 \varepsilon}} , \]
as claimed.

\section*{Acknowledgments} The author is supported by the Austrian Science Fund (FWF), project F-5510. I would like to thank Sary Drappeau for helpful discussions on Lemma \ref{BDlemma}.


\begin{thebibliography}{99}
\footnotesize{

\bibitem{BE1} J.\ Beck: \textit{Probabilistic Diophantine Approximation. Randomness in Lattice Point Counting.} Springer Monographs in Mathematics. Springer, Cham, 2014.

\bibitem{BE2} J.\ Beck: \textit{Randomness of the square root of 2 and the giant leap, part 1.} Period.\ Math.\ Hungar.\ 60 (2010), 137--242.

\bibitem{BE3} J.\ Beck: \textit{Randomness of the square root of 2 and the giant leap, part 2.} Period.\ Math.\ Hungar.\ 62 (2011), 127--246.

\bibitem{BD} S.\ Bettin and S.\ Drappeau: \textit{Limit laws for rational continued fractions and value distribution of quantum modular forms.} arXiv:1903.00457.

\bibitem{BI} D.\ Bilyk: \textit{The $L^2$ discrepancy of irrational lattices.} Monte Carlo and quasi-Monte Carlo methods 2012, 289--296, Springer Proc.\ Math.\ Stat., 65, Springer, Heidelberg, 2013.

\bibitem{BTY1} D.\ Bilyk, V.\ Temlyakov and R.\ Yu: \textit{Fibonacci sets and symmetrization in discrepancy theory.} J.\ Complexity 28 (2012), 18--36.

\bibitem{BTY2} D.\ Bilyk, V.\ Temlyakov and R.\ Yu: \textit{The $L_2$ discrepancy of two-dimensional lattices.} Recent advances in harmonic analysis and applications, 63--77, Springer Proc.\ Math.\ Stat., 25, Springer, New York, 2013.

\bibitem{BO1} B.\ Borda: \textit{On the distribution of Sudler products and Birkhoff sums for the irrational rotation.} arXiv:2104.06716.

\bibitem{BO2} B.\ Borda: \textit{On the theorem of Davenport and generalized Dedekind sums.} J.\ Number Theory 172 (2017), 1--20.

\bibitem{BO3} B.\ Borda: \textit{The number of lattice points in irrational polytopes.} Ph.D.\ Thesis, Rutgers The State University of New Jersey - New Brunswick. ProQuest LLC, 2016.

\bibitem{DA} H.\ Davenport: \textit{Note on irregularities of distribution.} Mathematika 3 (1956), 131--135.

\bibitem{DV} H.\ Diamond and J.\ Vaaler: \textit{Estimates for partial sums of continued fraction partial quotients.} Pacific J.\ Math.\ 122 (1986), 73--82.

\bibitem{DP} J.\ Dick and F.\ Pillichshammer: \textit{Digital Nets and Sequences. Discrepancy Theory and Quasi-Monte Carlo Integration.} Cambridge University Press, Cambridge, 2010.

\bibitem{DT} M.\ Drmota and R.\ Tichy: \textit{Sequences, Discrepancies and Applications.} Lecture Notes in Mathematics, 1651. Springer-Verlag, Berlin, 1997.

\bibitem{HE} L.\ Heinrich: \textit{Rates of convergence in stable limit theorems for sums of exponentially $\psi$-mixing random variables with an application to metric theory of continued fractions.} Math.\ Nachr.\ 131 (1987), 149--165.

\bibitem{HKP} A.\ Hinrichs, R.\ Kritzinger and F.\ Pillichshammer: \textit{Extreme and periodic $L_2$ discrepancy of plane point sets.} Acta Arith.\ 199 (2021), 163--198.

\bibitem{IK} M.\ Iosifescu and C.\ Kraaikamp: \textit{Metrical Theory of Continued Fractions.} Mathematics and its Applications, 547. Kluwer Academic Publishers, Dordrecht, 2002.

\bibitem{KE} H.\ Kesten: \textit{The discrepancy of random sequences $\{ kx \}$.} Acta Arith.\ 10 (1964/65), 183--213.

\bibitem{MO} T.\ Morita: \textit{Local limit theorem and distribution of periodic orbits of Lasota--Yorke transformations with infinite Markov partition.} J.\ Math.\ Soc.\ Japan 46 (1994), 309--343.

\bibitem{PE} V.\ Petrov: \textit{Limit Theorems of Probability Theory. Sequences of Independent Random Variables.} Oxford Studies in Probability, 4. Oxford Science Publications. The Clarendon Press, Oxford University Press, New York, 1995.

\bibitem{RS} L.\ Ro\c{c}adas and J.\ Schoi{\ss}engeier: \textit{An explicit formula for the $L_2$-discrepancy of $(n \alpha)$-sequences.} Computing 77 (2006), 113--128.

\bibitem{RO} K.\ Roth: \textit{On irregularities of distribution.} Mathematika 1 (1954), 73--79.

\bibitem{SA} J.\ Samur: \textit{On some limit theorems for continued fractions.} Trans.\ Amer.\ Math.\ Soc.\ 316 (1989), 53--79.

\bibitem{SCH} J.\ Schoissengeier: \textit{Another proof of a theorem of J.\ Beck.} Monatsh.\ Math.\ 129 (2000), 147--151.

\bibitem{TA} H.\ Takahasi: \textit{Large deviations for denominators of continued fractions.} Nonlinearity 33 (2020), 5861--5874.

}
\end{thebibliography}
\end{document}